\documentclass[12pt, reqno]{amsart}
\makeatletter
\usepackage[margin=1in]{geometry}
\usepackage{amsmath, amssymb, amsfonts, amstext, verbatim, amsthm, mathrsfs}
\usepackage[mathcal]{eucal}
\usepackage{microtype}
\usepackage[all]{xy}

\usepackage[usenames]{color}
\usepackage{aliascnt} 
\usepackage[colorlinks=true,linkcolor=blue,citecolor=blue,urlcolor=blue,citebordercolor={0 0 1},urlbordercolor={0 0 1},linkbordercolor={0 0 1}]{hyperref} 
\usepackage{enumerate}
\usepackage{xspace}
\usepackage{tikz}

\theoremstyle{plain}

\newcommand{\refnewtheoremn}[4]{%
\newaliascnt{#1}{#2}
\newtheorem{#1}[#1]{#3}
\aliascntresetthe{#1}
\expandafter\providecommand\csname #1autorefname\endcsname{#4}}

\newcommand{\refnewtheorem}[3]{\refnewtheoremn{#1}{#2}{#3}{#3}}

\def\makeCal#1{%
\expandafter\newcommand\csname c#1\endcsname{\mathcal{#1}}}
\def\makeBB#1{%
\expandafter\newcommand\csname b#1\endcsname{\mathbb{#1}}}
\def\makeFrak#1{%
\expandafter\newcommand\csname f#1\endcsname{\mathfrak{#1}}}

\count@=0
\loop
\advance\count@ 1
\edef\y{\@Alph\count@}%
\expandafter\makeCal\y
\expandafter\makeBB\y
\expandafter\makeFrak\y
\ifnum\count@<26
\repeat

\newtheorem{thm}{Theorem}[section]

\refnewtheorem{prethm}{thm}{Pre-theorem}

\refnewtheorem{lem}{thm}{Lemma}
\refnewtheorem{claim}{thm}{Claim}
\refnewtheorem{cor}{thm}{Corollary}
\refnewtheorem{conj}{thm}{Conjecture}
\refnewtheorem{prop}{thm}{Proposition}
\refnewtheorem{quest}{thm}{Question}
\refnewtheorem{amplif}{thm}{Amplification}
\refnewtheorem{scholium}{thm}{Scholium}

\refnewtheorem{tablethm}{thm}{Table}
\refnewtheorem{assumption}{thm}{Assumption}
\refnewtheorem{conclusion}{thm}{Conclusion}
\refnewtheorem{problem}{thm}{Problem}
\refnewtheorem{ansatz}{thm}{Ansatz}

\theoremstyle{definition}
\refnewtheorem{rem}{thm}{Remark}
\refnewtheorem{defn}{thm}{Definition}
\refnewtheorem{notn}{thm}{Notation}
\refnewtheorem{const}{thm}{Construction}
\refnewtheorem{ex}{thm}{Example}
\refnewtheorem{fact}{thm}{Fact}
\refnewtheorem{recollection}{thm}{Recollection}

\newcommand{\op}{\operatorname}

\newcommand{\dual}{\vee}
\newcommand{\fg}{\mathfrak{g}}
\newcommand{\Crit}{\op{Crit}}
\newcommand{\Rep}{\op{Rep}}
\newcommand{\pair}[2]{\left\langle #1,#2 \right\rangle}
\newcommand{\ol}{\overline}
\newcommand{\arxiv}[1]{\href{http://arxiv.org/abs/#1}{{\tt arXiv:#1}}}
\newcommand{\Perf}{\op{Perf}}
\newcommand{\Spec}{\op{Spec}}
\newcommand{\RHom}{\op{RHom}}
\newcommand{\Irr}{\op{Irrep}}

\newcommand{\df}[1]{{\bf \textsf{#1}}}
\newcommand{\git}{/\!/}
\newcommand{\bv}{{\bf v}}
\newcommand{\bw}{{\bf w}}
\newcommand{\dgCat}{\mathrm{dg}\text{-}\mathrm{Cat}}

\begin{document}

\title{Combinatorial constructions of derived equivalences}
\author{Daniel Halpern-Leistner}
\address{Department of Mathematics, Columbia University}
\email{\href{mailto:danhl@math.columbia.edu}{danhl@math.columbia.edu}}
\urladdr{\url{http://math.columbia.edu/~danhl/}}

\author{Steven V Sam}
\address{Department of Mathematics, University of Wisconsin, Madison}
\email{\href{mailto:svs@math.wisc.edu}{svs@math.wisc.edu}}
\urladdr{\url{http://math.wisc.edu/~svs/}}

\thanks{DHL was partially supported by NSF DMS-1303960. SS was partially supported by NSF DMS-1500069.}

\date{June 22, 2016}

\begin{abstract}
Given a certain kind of linear representation of a reductive group, referred to as a quasi-symmetric representation in recent work of {\v S}penko and Van den Bergh, we construct equivalences between the derived categories of coherent sheaves of its various geometric invariant theory (GIT) quotients for suitably generic stability parameters. These variations of GIT quotient are examples of more complicated wall crossings than the balanced wall crossings studied in recent work on derived categories and variation of GIT quotients.

Our construction is algorithmic and quite explicit, allowing us to: 1) describe a tilting vector bundle which generates the derived category of such a GIT quotient, 2) provide a combinatorial basis for the $K$-theory of the GIT quotient in terms of the representation theory of $G$, and 3) show that our derived equivalences satisfy certain relations, leading to a representation of the fundamental groupoid of a ``K\"ahler moduli space'' on the derived category of such a GIT quotient. Finally, we use graded categories of singularities to construct derived equivalences between all Deligne--Mumford hyperk\"ahler quotients of a symplectic linear representation of a reductive group (at the zero fiber of the algebraic moment map and subject to a certain genericity hypothesis on the representation), and we likewise construct actions of the fundamental groupoid of the corresponding K\"ahler moduli space.
\end{abstract}

\maketitle

\setcounter{tocdepth}{1}
\tableofcontents

\section{Introduction}

One of the motivating conjectures in the theory of derived categories of coherent sheaves states that two smooth algebraic varieties which are projective over an affine variety $Y,Y' \to \Spec(R)$ which are K-equivalent -- meaning that they are birational to one another and the pullback of their canonical bundles to a resolution of this birational map are isomorphic -- have an equivalence between their derived categories of coherent sheaves ${\rm D}^b(Y) \simeq {\rm D}^b(Y')$.\footnote{One expects that this equivalence restricts to the identity  away from the exceptional locus of the flop.} In the case where $R = k$ is the ground field (always assumed to have characteristic $0$), the conjecture is settled for 3-folds \cite{Bridgeland}, as well as for toric varieties \cite{Kawamata}.

This conjecture originates in the work of Bondal and Orlov \cite{BondalOrlov}, and is inspired by homological mirror symmetry. In fact, a careful reading of the mirror symmetry philosophy predicts an action of the fundamental groupoid of a certain ``complexified K\"ahler moduli space'' on the derived category of a Calabi-Yau manifold. This paper settles the D-equivalence conjecture and constructs such an action of the fundamental groupoid of the complexified K\"ahler moduli space for a large class of varieties and orbifolds arising as GIT quotients of certain linear representations of a reductive group.

As any birational transformation between smooth projective varieties over $\Spec(R)$ can be modeled explicitly as a variation of geometric invariant theory (GIT) quotients (see \cite{HuKeel} for instance), it is natural to approach the D-equivalence conjecture from this perspective, and in fact there is a general framework for doing so, established in \cite{HL,BFK} following \cite{HHP,Segal}, which we now recall.

Given a smooth projective-over-affine variety $X$ with a reductive group $G$ acting on $X$, any choice of $G$-ample line bundle $L$ defines a $G$-equivariant open semistable locus $X^{\rm{ss}}(L) \subset X$, and we refer to the quotient stack $X^{\rm{ss}}/G$ as the GIT quotient.\footnote{One is typically interested in the situation when $X^{\rm{ss}}/G$ is a scheme, but our methods apply just as well when $G$ acts on $X^{\rm{ss}}(L)$ with finite stabilizers and therefore $X^{\rm{ss}}(L)/G$ is a Deligne--Mumford stack. In the language of GIT this is the condition that $X^{\rm{ss}}(L) = X^{\rm{s}}(L)$.} Two $G$-ample bundles $L_\pm$ lead to birational stacks $X^{\rm{ss}}(L_\pm) / G$, and when this birational transformation is a K-equivalence, the method for establishing a derived equivalence ${\rm D}^b(X^{\rm{ss}}(L_+) /G) \simeq {\rm D}^b(X^{\rm{ss}}(L_-)/G)$ is to verify the following:

\begin{ansatz} \label{ans:G}
There is a full subcategory $\cG \subset {\rm D}^b(X/G)$ such that the restriction functor gives an equivalence $\op{res}_\pm \colon \cG \xrightarrow{\simeq} {\rm D}^b(X^{\rm{ss}}(L_\pm) / G)$.
\end{ansatz}

If one can find such a category $\cG$, then the equivalence is simply $\op{res}_- \circ \op{res}_+^{-1}$. For a smooth projective-over-affine $X$ with an action of a reductive group $G$, the main theorems of \cite{HL} and \cite{BFK} produce, for any GIT quotient, a category $\cG$ which is identified with ${\rm D}^b(X^{\rm{ss}}/G)$ under restriction -- the category $\cG$ is described as the full subcategory of complexes satisfying a local weight condition (see \S \ref{sect:magic_windows}). Unfortunately, it is not always the case that the category produced for $X^{\rm{ss}}(L_+)/G$ will coincide with the category for $X^{\rm{ss}}(L_-)/G$ as in \autoref{ans:G}. Prior to this paper, the only general statement to this effect applies to the simplest kind of variation of GIT, referred to as a ``balanced wall crossing'' in \cite{HL} and an ``elementary wall crossing'' in \cite{BFK}. This is enough to establish many new cases of the conjecture, but there are many more variations of GIT quotient in which derived equivalences are expected, but not yet established.

In this paper we focus on the ``local'' version of this conjecture, where $Y,Y' \to \Spec(R)$ are birational maps. Our investigation began with an attempt to understand a series of examples where derived equivalences have been established for more complicated variations of GIT quotient, and specifically the examples of \cite{DonovanSegal}, \cite{DonovanSegal2}.\footnote{Similar questions have been investigated by Paul Horja, who presented some results on the derived category of GIT quotients of linear representations of tori in a talk in 2012 at the IPMU in Kashiwa, Japan in which certain zonotopes played a key role.} Those papers also verify \autoref{ans:G}, but for a category defined explicitly by a set of generating vector bundles. We will show that this phenomenon happens in a much wider class of examples. We focus on the local model: we consider a linear representation $X$ of a reductive group $G$, and we assume that $X$ is \emph{quasi-symmetric} (see the definition in \S\ref{sect:zonotope}). In particular, our methods apply to all representations which are self-dual.

We have chosen to focus on the case of GIT quotients of a linear representation because it greatly simplifies the exposition and highlights the combinatorial and concrete aspects of the construction. It is also closer in spirit to the work in the physics literature on the gauged linear sigma model \cite{HHP,HHP2}, further explained in mathematical terms in \cite{Segal}, which introduced \autoref{ans:G} as a method for producing derived equivalences between different GIT quotients. It is important to note that although the balanced wall crossings studied in \cite{HL} establish many new instances of the D-equivalence conjecture, every balanced wall-crossing is \'etale locally equivalent to a wall crossing of the kind initially studied in \cite{Segal}. In contrast, the wall-crossings studied here constitute a fundamentally new set of examples, generalizing those of \cite{DonovanSegal,DonovanSegal2}. As with \cite{Segal}, the linear case studied here serves as an \'{e}tale local model for ``locally quasi-symmetric'' variations of GIT quotient more generally. In follow-up work, we will use our main theorem to establish a ``global'' statement which generalizes many if not all of the known instances in which  a derived equivalence arises from a variation of GIT quotient (see Section \ref{sect:non_local} for a brief discussion).

\bigskip

Our main theorem defines a certain Weyl-invariant polytope $\ol{\nabla} \subset M_\bR$, where $M$ is the character lattice of a maximal torus $T \subset G$. For any $\delta \in M_\bR^W$ we define $\cM(\delta + \ol{\nabla}) \subset {\rm D}^b(X/G)$ to be the full subcategory generated by vector bundles of the form $\cO_X \otimes U$, where $U$ is an irreducible representation of $G$ whose character lies in $\delta + \ol{\nabla}$, a ``magic window.'' Likewise, characters of the group $\ell \in \op{Pic}(BG)$ determine $G$-linearized invertible sheaves $\cO_X \otimes \ell$ with which to form a GIT quotient. It is known \cite{DolgachevHu} that $X^{\rm{ss}}(\cO_X \otimes \ell)$ only depends on the class of $\ell$ in $\op{Pic}(BG)\otimes \bR \cong M_\bR^W$ and in fact only on which ``cell'' $\ell$ lies in with respect to a rational wall and chamber decomposition of $M_\bR^W$, so it is standard to regard the semistable locus $X^{\rm{ss}}(\ell)$ as determined by an arbitrary $\ell \in M_\bR^W$. Our main theorem states:

\begin{thm} [\autoref{thm:magic_windows}]
If $X$ is a quasi-symmetric representation of $G$ satisfying a suitable genericity condition, then for any $\ell \in M_\bR^W$ such that $X^{\rm{ss}}(\ell) = X^{\rm{s}}(\ell)$ and any $\delta \in M_\bR^W$ such that $M \cap \partial(\delta + \ol{\nabla}) = \emptyset$, the restriction functor
\[
\cM(\delta + \ol{\nabla}) \to {\rm D}^b(X^{\rm{ss}}(\ell)/G)
\]
is an equivalence of dg-categories.
\end{thm}

The ``genericity'' condition on $X$ is the requirement that if one considers the GIT quotient of $X$ by the action of the maximal torus $T\subset G$, then $X^{T\text{\rm{-ss}}}(\ell) = X^{T\text{\rm{-s}}}(\ell)$ for some $\ell \in M_\bR^W$. In \S\ref{sect:zonotope} we explicitly identify a linear hyperplane arrangement in $M_\bR$ such that the genericity condition holds if and only if none of these hyperplanes contain $M_\bR^W$. We identify the GIT chambers for the action of $T$ on $X$ in the sense of \cite{DolgachevHu} with the connected components of the complement of this hyperplane arrangement.

\begin{cor}[\autoref{cor:equivalences}]
Under the hypotheses of \autoref{thm:magic_windows}, if $\ell,\ell' \in \op{Pic}(BG)_\bR$ are such that $X^{\rm{ss}}=X^{\rm{s}}$, then we have an equivalence ${\rm D}^b(X^{\rm{ss}}(\ell)/G) \simeq {\rm D}^b(X^{\rm{ss}}(\ell') /G)$.
\end{cor}

Another application of \autoref{thm:magic_windows} is to construct an explicit and efficiently computable basis for the algebraic $K$-theory of the GIT quotient $X^{\rm{ss}}(\ell)/G$. In \autoref{prop:K_theory} we show that the algebraic $K$-theory of $X^{\rm{ss}}(\ell)/G$ is a free $\bZ$-module and has a basis of locally free sheaves induced from those irreducible representations of $G$ whose character lies in $\delta + \ol{\nabla}$. When $k=\bC$, the same is true for the topological $K$-theory and orbifold cohomology of the analytification of $X^{\rm{ss}}(\ell)/G$. By providing an explicit basis in terms of the representation theory of $G$, we refine the classic work of \cite{EllingsrudStromme}, which provides an explicit presentation of the Chow ring of $X^{\rm{ss}}/G$.

Our methods are closely related to the approach to derived equivalences initiated in \cite{VdBFlops}, where the derived category of two different resolutions of a singularity are identified with that of a common non-commutative resolution. In fact, the ``essential surjectivity'' part of \autoref{thm:magic_windows} builds on the methods of the recent preprint \cite{SVdB}, which provides an explicit algorithm for constructing non-commutative resolutions of the singular affine scheme $\Spec(\cO_X^G)$. The original algorithm of \cite{SVdB} involves starting with an equivariant vector bundle on $X$, and then taking iterated mapping cones with certain special complexes (called $C_{\lambda,\chi}$) until it lies in the category $\cM(\delta + \ol{\nabla})$. This algorithm can modify the complex over $X^{\rm{ss}}$, so it does not exactly suit the purpose of showing that every complex on ${\rm D}^b(X^{\rm{ss}}/G)$ is generated by the restriction of complexes in $\cM(\delta + \ol{\nabla})$. One of our key observations is that one can modify the algorithm of \cite{SVdB} without significantly modifying the combinatorial scaffolding in order to incorporate the interaction with GIT. The upshot is the construction (see \autoref{cor:tilting}) of a locally free sheaf on $X/G$ which generates $\cM(\delta + \ol{\nabla})$ and whose restriction to $X^{\rm{ss}}(\ell)/G$ is a tilting generator for ${\rm D}^b(X^{\rm{ss}}(\ell)/G)$. The algebra of endomorphisms of this tilting generator is precisely the non-commutative resolution constructed in \cite{SVdB}.

To date, the most progress on the local D-equivalence conjecture has been made in the case when $Y \to \Spec(R)$ is a symplectic resolution, meaning that $\Spec(R)$ is normal and there is a non-degenerate closed 2-form in ${\rm H}^0(Y,\Omega^2_Y)$. The conjecture was solved for resolutions of symplectic finite quotient singularities in \cite{BezrukavnikovKaledin}, and for symplectic resolutions for which $\Spec(R)$ admits a $\bG_m$-action with ``positive weights'' in \cite{Kaledin},\footnote{Technically this result is not stated in \cite{Kaledin}, but it follows from a short argument using the methods and results there.} using the theory of Fedosov quantization in characteristic $p>0$ to construct tilting bundles which lift to characteristic $0$. We recover this result by entirely different methods for symplectic resolutions which arise as hyperk\"ahler reductions of a symplectic linear representation $X$ of a reductive group $G$. In this case, there is a canonical $G$-equivariant algebraic moment map $\mu \colon X \to \fg^\dual$. We define $X_0 = \mu^{-1}(0)$, and we define the hyperk\"ahler reduction to be $X_0^{\rm{ss}}(\ell) / G$.

\begin{thm}[\autoref{thm:hyperkaehler}]
Let $X$ be a symplectic linear representation of a reductive group $G$ such that the quasi-symmetric representation $X \oplus \fg$ satisfies the genericity condition of \autoref{thm:magic_windows}. Then for any pair of $\ell, \ell' \in \op{Pic}(BG)_\bR$ such that $X_0^{\rm{ss}} = X_0^{\rm{s}}$, we have a derived equivalence ${\rm D}^b(X_0^{\rm{ss}}(\ell)/G) \to {\rm D}^b(X_0^{\rm{ss}}(\ell')/G)$.
\end{thm}

The theorem extends the results of Kaledin to new examples of symplectic resolutions which are Deligne--Mumford stacks, but the real novelty of our approach is that it is more geometric and explicit. In particular, it clarifies the dependence of the derived equivalence for hyperk\"ahler flops in \autoref{thm:hyperkaehler} and for the class of flops in \autoref{cor:equivalences} on a certain path in the space of K\"{a}hler parameters, thereby allowing us to construct an action of the fundamental groupoid of this space on the corresponding derived category.

\subsection{Actions of fundamental groupoids on derived categories}

Let us recall the cartoon sketch of why homological mirror symmetry predicts that two birational Calabi--Yau manifolds $Y$ and $Y'$ should have equivalent derived categories: if $Y$ possesses a homological mirror partner $Y^\circ$, so that ${\rm D}^b(Y)$ is equivalent to the Fukaya category ${\rm DFuk}(Y^\circ)$, then any two points in the moduli space of complex structures on $Y^\circ$ correspond to two different K\"ahler manifolds (with the same underlying symplectic manifold), and any path between these two points in the moduli space leads to an equivalence between the Fukaya categories of these two K\"ahler manifolds via symplectic parallel transport. Mirror symmetry also identifies the moduli space of complex structures on $Y^\circ$ with the ``K\"ahler moduli space'' of $Y$, an open subset $\cK_Y \subset {\rm H}^{1,1}(Y) / 2\pi i \op{NS}(Y)$. So transporting our path in the complex moduli space of $Y^\circ$, we may find that the path in ${\rm H}^{1,1}(Y)$ has left the K\"ahler cone of $Y$ and entered the K\"ahler cone of $Y'$. It is then expected that $Y^\circ$, with this new complex structure, is a homological mirror partner to $Y'$, giving the equivalence ${\rm D}^b(Y) \simeq {\rm DFuk}(Y^\circ) \simeq {\rm D}^b(Y')$. This equivalence depends on the homotopy class of the path between the corresponding points in $\cK_Y$, and therefore one finds an action of the fundamental groupoid of $\cK_Y$ on ${\rm D}^b(Y)$.

Our final application of \autoref{thm:magic_windows} is to fully realize this K\"{a}hler moduli space picture in the examples we are studying. Given a polytope $\ol{\nabla} \subset M_\bR$, we define a locally finite $M^W$-periodic hyperplane arrangement $\cA = \{ H_\alpha \subset M_\bR^W \}$ characterized by the property that $\partial(\delta + \ol{\nabla}) \cap M = \emptyset$ if and only if $\delta$ lies in the complement of this hyperplane arrangement (\autoref{lem:hyperplane}). Our version of the ``complexified K\"ahler moduli space'' of the GIT quotient $X^{\rm{ss}}(\ell)/G$ is the space
\[
\cK_X := (M_\bC^W \setminus {\bigcup}_\alpha H_\alpha) / M^W.
\]
In \autoref{prop:groupoid} we give a presentation for the fundamental groupoid $\Pi_1(\cK_X)$ of $\cK_X$, and in \autoref{prop:representation} we construct a representation of $\Pi_1(\cK_X)$ in the homotopy category of dg-categories, i.e., a functor $F \colon \Pi_1(\cK_X) \to \op{Ho}(\dgCat)$, mapping any point of the form $\delta + 0 i \in M_\bC^W \setminus \bigcup_\alpha H_\alpha$ to the category $F([\delta]) = \cM(\delta + \ol{\nabla})$, which, by \autoref{thm:magic_windows}, is canonically identified under restriction with any ${\rm D}^b(X^{\rm{ss}}(\ell)/G)$ for which $X^{\rm{ss}}(\ell) = X^{\rm{s}}(\ell)$. We also extend our results to the case of the hyperk\"ahler quotient $X_0^{\rm{ss}}(\ell)/G$ for a symplectic linear representation $X$ of $G$ in \autoref{cor:representation_hyperkahler} -- in this case we obtain an action of $\Pi_1(\cK_{X\oplus \fg})$ on ${\rm D}^b(X_0^{\rm{ss}}(\ell)/G)$.

Actions of fundamental groupoids of this form have been intensively studied in connection with geometric representation theory. For instance, Cautis, Kamnitzer, and Licata \cite{CKL} have established generalized braid group actions on the derived category of certain Nakajima quiver varieties using very different methods. There the derived categories of several different quiver varieties form the ``weight spaces'' for a certain categorified representation of a Lie algebra, and the action of the generalized braid group is described implicitly via this categorical Lie algebra action \cite{CautisKamnitzer}. Even for quiver varieties our results are new, as \cite{CKL} only treats quiver varieties corresponding to simply-laced graphs (i.e., no loops or multiple edges); see \S\ref{sect:hilb} for a discussion of what our methods produce for the Hilbert scheme of $n$ points on $\bC^2$, the quiver variety associated to the quiver with one vertex and one loop.

Our results are also closely related to a conjectural picture developed by Bezrukavnikov and Okounkov, as described for instance in \cite[Conjecture 1]{ABM}. Their work, along with their coauthors, proposes a vast generalization of Kazhdan--Lusztig theory, whose starting point is the conjectural action of certain fundamental groupoids of complexified K\"{a}hler moduli spaces on the derived categories of symplectic resolutions. These are established for the Springer resolution and Slodowy slices in \cite{KT,BMR, BezrukavnikovRiche}, and recently for the cotangent bundles of partial flag varieties \cite{Boger}. The formulation of the conjecture involves quantization in characteristic $p$, which in principle gives an action of the fundamental groupoid of the complement of the complexification of some periodic real hyperplane arrangement, but at the moment the structure of that arrangement is still conjectural outside of the most basic examples.

Our method essentially substitutes the characteristic $0$ categories ${\rm D}_{\op{sing}}(\cM(\delta + \ol{\nabla}))$ for the quantizations in characteristic $p$, allowing us to completely and explicitly describe the groupoid actions conjectured by Bezrukavnikov and Okounkov for Deligne--Mumford stacks which arise as hyperk\"ahler reductions of symplectic linear representations. Our methods also apply to GIT quotients which are not algebraic symplectic, implying an intriguing extension of this story beyond the setting of symplectic resolutions. We hope that our combinatorial construction of these groupoid actions will advance the broader investigation into generalizations of Kazhdan--Lusztig theory, opening research into natural questions such as how to generalize the Hecke algebra and the theory of canonical bases into this context.

\subsection{Authors' note}

\autoref{thm:magic_windows} is a key input to an ongoing project of the first author with Davesh Maulik and Andrei Okounkov, and he gratefully acknowledges their guidance and encouragement, as well as many useful conversations about symplectic resolutions and quantization. The main theorem above will be used in follow-up work to construct ``stable envelope functors" and to categorify the ideas of \cite{MaulikOkounkov} for algebraic symplectic varieties.

The first author would also like to thank Ed Segal for many enlightening conversations over the years, and for explaining the methods of \cite{SVdB} and encouraging us to pursue the actions of fundamental groupoids.

\subsection{Notation} \label{sec:notation}

We fix a base field $k$ of characteristic $0$, and we work over this field throughout the paper. $G$ will denote a split reductive group over $k$, and $X$ will denote a linear representation of $G$ over $k$. Let $\beta_1, \dots, \beta_d$ be the weights of $X$ (counted with multiplicity). 

We fix once and for all a maximal torus and Borel subgroup $T \subset B \subset G$. We denote the character lattice of $T$ by $M$, and the cocharacter lattice by $N$. We denote the Weyl group by $W$. We fix a choice of $W$-invariant inner product $\pair{}{}$ on $M$ and $N$. We use the convention that the weights in the Lie algebra of $B$ are {\it negative} roots; this is consistent with \cite{SVdB}. In particular, the choice of $B$ determines a choice of dominant chamber $M^+_\bR \subset M_\bR$. We let $\rho$ denote the sum of the positive roots of $G$ divided by $2$.

Given a dominant weight $\chi$ of a connected group $G$, we let $V(\chi)$ denote the irreducible representation of $G$ of highest weight $\chi$. We let $\Irr(G)$ denote the set of irreducible representations of $G$, and we let $\op{Char}(U) \subset M$ denote the set of non-zero weights appearing in a representation $U$. We will denote the group of characters by $\op{Pic}(BG)$, and we will often consider the real vector space spanned by characters  $\op{Pic}(BG)_\bR := \op{Pic}(BG)\otimes_\bZ \bR$. This is canonically identified with the Weyl-invariant subspace $\op{Pic}(BG)_\bR = M_\bR^W \subset M_\bR$.

\section{Zonotopes and Deligne--Mumford GIT quotients} \label{sect:zonotope}

In this section we consider a \df{quasi-symmetric} linear representation $X$ of a reductive group $G$. By definition \cite{SVdB} this means that if we let $\beta_i \in M$ for $i=1,\ldots,d$ denote the $T$-weights of the representation $X^\dual$, indexed with repetitions according to the dimension of the corresponding weight space, then for any line $L \subset M_\bR$ we have $\sum_{\beta_i \in L} \beta_i = 0$. Note that any self-dual representation is quasi-symmetric.

In \cite{SVdB}, {\v S}penko and Van den Bergh also consider the following convex regions in $M_\bR$:
\begin{align*}
\ol{\Sigma} &:= \left\{ \left. {\sum}_i a_i \beta_i\, \right|\, a_i \in [-1,0]  \right\}, \\
\ol{\Sigma}_\varepsilon &:= \bigcup_{r>0} \ol{\Sigma} \cap (r\varepsilon + \ol{\Sigma}).
\end{align*}
Note that $\ol{\Sigma}$ can alternatively be described as the Minkowski sum of the intervals $[-\beta_i, 0]$ (an object known as a \df{zonotope}), or as the convex hull of the character of the exterior algebra $\bigwedge^\ast X \in \Rep(G)$. The fact that $X$ is quasi-symmetric implies that $\ol{\Sigma}$ can also be described as the Minkowski sum $\sum_i [0,\beta_i]$, as can be verified by breaking the Minkowski sum into the partial sums $\sum_{\beta_i \in L} [0,\beta_i]$ for each line $L \subset M_\bR$. Thus $\ol{\Sigma} = -\ol{\Sigma}$, and $\ol{\Sigma}$ is the convex hull of the character of $\bigwedge^\ast (X^\dual)$. Note also that $\ol{\Sigma}_\varepsilon$ is just $\ol{\Sigma}$ with some of its boundary faces removed. Following \cite{SVdB}, we say that $\ell \in M_\bR$ is \df{generic} for $\ol{\Sigma}$ if it lies in the linear span of the points of $\ol{\Sigma}$ but is not parallel to any face of $\ol{\Sigma}$.

Recall that given a character $\ell \in \op{Pic}(BG)$, we can define a $G$-equivariant open semistable locus $X^{\rm{ss}}(\ell) \subset X$ by
\[
X^{\rm{ss}}(\ell) = \{x \in X \mid \exists k>0 \text{ and } s \in \Gamma(\cO_X \otimes (k \ell))^G \text{ such that } s(x) \neq 0\}.
\]
$X^{\rm{ss}}(\ell)/G$ admits a good quotient $X^{\rm{ss}}(\ell) \git G$, which is proper over the affine quotient $\Spec(\cO_X^G)$. We recall that the \emph{Hilbert--Mumford criterion} for semistability states that a point $x \in X$ is semistable if and only if $\pair{\lambda}{\ell}<0$ for all one-parameter subgroups $\lambda \colon \bG_m \to G$ for which $\lim_{t\to 0} \lambda(t) \cdot x$ exists. When points of $X^{\rm{ss}}(\ell)$ have finite stabilizers, then $X^{\rm{ss}}(\ell)/G$ is a smooth Deligne--Mumford stack, and $X^{\rm{ss}}(\ell)/G \to \Spec(\cO_X^G)$ is a resolution of singularities. As remarked in the introduction, the Hilbert--Mumford criterion allows us to define $X^{\rm ss}(\ell)$ for any $\ell \in \op{Pic}(BG)_\bR$.

We can regard any $G$-representation as a $T$-representation and restrict a character $\ell$ to $\op{Pic}(BT)_\bR = M_\bR$. We denote the semistable locus with respect to the $T$ action as $X^{T\text{\rm{-ss}}}(\ell)$. One of the main observations of this paper is the following:

\begin{prop} \label{prop:semistable}
Let $X$ be a quasi-symmetric representation of a reductive group $G$ such that the action of $T$ on $X$ has generically finite stabilizers, and let $\ell \in \op{Pic}(BG)_\bR$ be a character. Then the following are equivalent:
\begin{enumerate}[\indent \rm (1)]
\item the GIT quotient $X^{T\text{\rm{-ss}}}(\ell) / T$ of $X$ by $T$ is Deligne--Mumford,
\item for any proper subspace $V \subsetneq M_\bR$, there is a one-parameter subgroup $\lambda$ such that $\pair{\lambda}{\beta_i} = 0$ for all $\beta_i \in V$ and $\pair{\lambda}{\ell} \neq 0$, and
\item $\ell$ is generic for $\ol{\Sigma}$.
\end{enumerate}
These conditions imply that $X^{\rm{ss}}(\ell)/G$ is Deligne--Mumford.
\end{prop}

\begin{proof}
\noindent $(1) \Rightarrow (2)$:
Choose a subspace $V \subsetneq M_\bR$ and consider a generic point $x \in \sum_{\chi \in V} X_{\chi} \subset X$, which will have non-vanishing coordinates in the eigenspace for each weight $\beta_i$ such that $\beta_i \in V$. Then $x$ has a positive dimensional stabilizer. Since $X^{T\text{\rm{-ss}}}(\ell)/T$ is Deligne--Mumford, there must be a $\lambda \in N$ which destabilizes $x$. By quasi-symmetry, the condition that $\pair{\lambda}{\beta_i} \geq 0$ for all $\beta_i \in V$ actually implies that $\pair{\lambda}{\beta_i} = 0$ for all $\beta_i \in V$, hence we have $(2)$.

\medskip
\noindent $(2) \Rightarrow (3)$:
Every face of $\ol{\Sigma}$ is a Minkowski sum of the form
\[
F = a_1 \beta_{\sigma(1)} + \cdots + a_k \beta_{\sigma(k)} + [-\beta_{\sigma(k+1)},0] + \cdots + [-\beta_{\sigma(d)},0],
\]
where $a_k \in \{-1,0\}$ and $\sigma$ is some permutation of the set $\{1,\ldots,d\}$ \cite{mcmullen}. So after translation $F$ lies in the subspace spanned by $\beta_{\sigma(k+1)},\ldots,\beta_{\sigma(d)}$. To prove (3), we may assume $F$ spans a proper subspace $V \subsetneq M_\bR$. The existence of a cocharacter vanishing on $V$ but not on $\ell$ implies that $\ell$ is not parallel to $F$. $T$ acts with generically finite stabilizers on $X$ if and only if $\op{Char}(X)$ spans $M_\bR$, and in particular this implies that the linear span of $\ol{\Sigma}$ is all of $M_\bR$ and includes $\ell$.

\medskip
\noindent $(3) \Rightarrow (1)$:
This follows from \autoref{thm:magic_windows} below (which only uses the implications $(1) \Rightarrow (2) \Rightarrow (3)$ from the current proof), which states that ${\rm D}^b(X^{T\text{\rm{-ss}}}(\ell)/T)$ is generated by a finite set of vector bundles if $\ell$ is generic for $\ol{\Sigma}$. If $X^{T\text{\rm{-ss}}}/T$ were not Deligne--Mumford, then there would be a closed point $x \in X^{T\text{\rm{-ss}}}(\ell)$ stabilized by a sub-torus $T' \subset T$. The map ${x}/T' \to X^{T\text{\rm{-ss}}}(\ell)/T$ is affine, which implies that any generating set for ${\rm D}^b(X^{T\text{\rm{-ss}}}(\ell)/T)$ pulls back to a generating set for ${\rm D}^b({x}/T')$. The latter does not admit a finite generating set of perfect complexes, so we see that $X^{T\text{\rm{-ss}}}/T$ must be Deligne--Mumford.

\medskip
\noindent \textit{Showing $X^{\rm{ss}}/G$ is Deligne--Mumford}:
Every fiber of the map $X^{\rm{ss}}(\ell)/G \to X^{\rm{ss}}(\ell) \git G$ contains a point whose stabilizer group is reductive, so $X^{\rm{ss}}(\ell)/G$ is Deligne--Mumford if and only if every point $x \in X$ with a $\bG_m$ in its stabilizer group is unstable. By replacing $x$ with $g\cdot x$ for some $g\in G$, we may assume that $x$ has a positive dimensional stabilizer in $T$, and because all $G$-semistable points are $T$-semistable we thus can exhibit a point which is $T$-semistable and has positive dimensional stabilizer in $T$. This reduces the claim to the case where $G=T$.
\end{proof}

\subsection{Wall-and-chamber decompositions in the quasi-symmetric case}

For any linear representation $X$, the space $M_\bR^W$ is canonically identified with the equivariant N\'eron--Severi group $\op{NS}^G(X) \otimes \bR$ studied in \cite{DolgachevHu}. There, Dolgachev and Hu construct a finite collection of closed subsets of $M_\bR^W$ called walls, each a union of finitely many rational polyhedral cones, such that the connected components of the complements of these walls are precisely the characters of $G$ such that $X^{\rm ss}(\ell) = X^{\rm s}(\ell)$, and $X^{\rm ss}(\ell)$ is constant for $\ell$ in each chamber. The closure of each chamber is also a rational polyhedral cone.

In general, this wall-and-chamber decomposition can be somewhat complicated, but \autoref{prop:semistable} leads to some simplifications under the assumptions that 1) $X$ is a quasi-symmetric, and 2) the generic stabilizer for the action of $T$ on $X$ is finite. The second assumption is equivalent to $\ol{\Sigma}$, the zonotope associated to the character of $X$, linearly spanning $M_\bR$.

\begin{defn}[Wall and chamber decomposition]
Assuming $\ol{\Sigma}$ spans $M_\bR$, consider the linear hyperplane arrangement $\cA = \{H_\alpha \subset M_\bR\}$ in $M_\bR$ consisting of the linear subspaces parallel to the codimension $1$ faces (facets) of $\ol{\Sigma}$. We denote the set of points of $M_\bR^W$ which are generic for $\ol{\Sigma}$ by
\[
(M_\bR^W)_{\ol{\Sigma}\text{-{\rm gen}}} = M_\bR^W \setminus \left({\bigcup}_\alpha H_\alpha \cap M_\bR^W \right).
\]
\end{defn}

Observe that for any non-trivial one-parameter subgroup $\lambda : \bG_m \to G$, one obtains a face $F_\lambda$ of $\ol{\Sigma}$ by
\[
F_\lambda = \left\{ \left. \sum_{\langle \lambda, \beta_i\rangle > 0} -\beta_i + \sum_{\langle \lambda, \beta_i\rangle = 0} a_i \beta_i \right| a_i \in [-1,0] \right\},
\]
and every face arises in this way. For any codimension $1$-subspace $H$ which is spanned by a subset of the $\beta_i$, there is a non-trivial one-parameter subgroup vanishing on these weights, and it follows that $H$ is parallel to a codimension $1$ face of $\ol{\Sigma}$, i.e. $H \in \cA$. On the other hand if $F_\lambda$ is codimension $1$, then the $\beta_i$ such that $\langle \lambda, \beta_i \rangle = 0$ must span a codimension $1$ subspace. So $\cA$ consists of precisely those linear hyperplanes spanned by a subset of the $\beta_i$.

As an immediate consequence of \autoref{prop:semistable}, we have:

\begin{cor}
If the generic stabilizer of $T$ on $X$ is finite, then the chambers of $M_\bR \cong \op{NS}^T(X)_\bR$ arising from GIT for $X/T$ are \emph{exactly} the connected components of the complement of the linear hyperplane arrangement $\cA$ above.
\end{cor}

Most of our results will require the hypothesis that $(M_\bR^W)_{\ol{\Sigma}\text{-{\rm gen}}} \neq \emptyset$. Although the generic locus of $M_\bR$ is always open and dense (still assuming that $\ol{\Sigma}$ spans $M_\bR$), it can unfortunately happen that $(M_\bR^W)_{\ol{\Sigma}\text{-{\rm gen}}} = \emptyset$. Note that by definition $(M_\bR^W)_{\ol{\Sigma}\text{-{\rm gen}}} = \emptyset$ if and only if $M_\bR^W \subset H_\alpha$ for some hyperplane $H_\alpha \in \cA$, or equivalently if there is a subset of weights $\beta_{i_1},\ldots,\beta_{i_n}$ such that
\[
M_\bR^W \subset \op{Span}(\beta_{i_1},\ldots,\beta_{i_n}) \neq M_\bR.
\]

\begin{ex}
The representation $X = T^* {\rm Sym}^p (k^2)$ of ${\bf GL}_2(k)$ satisfies the condition $(M_\bR^W)_{\ol{\Sigma}\text{-{\rm gen}}} \neq \emptyset$ if and only if $p$ is odd. More generally, if $M_\bR^W$ is spanned by the Weyl invariant weights of $X$, then $(M_\bR^W)_{\ol{\Sigma}\text{-{\rm gen}}} = \emptyset$.
\end{ex}

\autoref{prop:semistable} implies that $X^{G\text{\rm{-ss}}}(\ell) = X^{G\text{\rm{-s}}}(\ell)$ for all $\ell \in (M_\bR^W)_{\ol{\Sigma}\text{-{\rm gen}}}$. It follows that each connected component of $(M_\bR^W)_{\ol{\Sigma}\text{-{\rm gen}}}$ is contained in a single GIT chamber of $M_\bR^W$ for the action of $G$ on $X$, and that for each GIT chamber of $(M_\bR^W)_{\ol{\Sigma}\text{-{\rm gen}}}$ there is a finite set of connected components of $(M_\bR^W)_{\ol{\Sigma}\text{-{\rm gen}}}$ whose union is an open dense subset of that chamber. As a consequence, we have

\begin{cor} \label{cor:perturbation}
Assume that $(M_\bR^W)_{\ol{\Sigma}\text{-{\rm gen}}} \ne \emptyset$, then for any $\ell \in M_\bR^W$
\begin{itemize}
\item one can find an $\ell'$ arbitrarily close to $\ell$ such that $X^{G\text{\rm{-ss}}}(\ell') = X^{G\text{\rm{-s}}}(\ell')$, and
\item if $X^{G\text{\rm{-ss}}}(\ell) = X^{G\text{\rm{-s}}}(\ell)$, one can find an $\ell'$ arbitrarily close to $\ell$ with $X^{G\text{\rm{-ss}}}(\ell') = X^{G\text{\rm{-ss}}}(\ell)$ such that $X^{T\text{\rm{-ss}}}(\ell') = X^{T\text{\rm{-s}}}(\ell')$.
\end{itemize}
\end{cor}

Thus we see that in the quasi-symmetric case when $(M^W_\bR)_{\ol{\Sigma}\text{-{\rm gen}}}\ne \emptyset$, we avoid the pathology of codimension $0$ walls, which occur for more general GIT quotients of linear representations. When there are codimension $0$ walls, there are open regions in $M_\bR^W$ for which $X^{\rm{ss}}(\ell) \neq X^{\rm{s}}(\ell)$.

\subsection{The polytope $\ol{\nabla}$}

Define $\bL = [X^\dual] - [\fg^\dual] \in K_0(\op{Rep}(T))$. For any one-parameter subgroup $\lambda$ of $G$, we define $\bL^{\lambda >0}$ to be the projection of this class onto the subspace spanned by weights which pair positively with $\lambda$. For any cocharacter $\lambda$ we define
\[
\eta_\lambda := \langle \lambda, \bL^{\lambda>0} \rangle.
\]
This agrees with the $\eta$ defined in \cite{HL}. We note also that $\eta_\lambda = \eta_{w \lambda}$ for any $w \in W$, and $\eta_\lambda = \eta_{-\lambda}$ because $X$ is quasi-symmetric.

\begin{defn}
Given a character $\varepsilon \in \op{Pic}(BG)_\bR$ we define the following subsets of $M_\bR$:
\[
\
\nabla_\varepsilon := \left\{ \chi \in M_\bR\, \left|\, -\frac{\eta_\lambda}{2} < \langle\lambda,\chi\rangle \leq \frac{\eta_\lambda}{2}\text{ for all } \lambda \colon \bG_m \to T \text{ such that } \langle\lambda,\varepsilon\rangle > 0 \right. \right\}.
\]
\end{defn}

\begin{rem} \label{rmk:nabla-width}
As a consequence of the symmetry of these defining inequalities around $0$, and the elementary fact that any cocharacter $\lambda$ with $\pair{\lambda}{\varepsilon} \geq 0$ can be approximated rationally by a $\lambda'$ with $\pair{\lambda'}{\varepsilon}>0$, the closure of $\nabla_\varepsilon$, which we denote $\ol{\nabla}$, is independent of $\varepsilon$:
\[
\ol{\nabla} := \left\{ \chi \in M_\bR\, \left|\, -\frac{\eta_\lambda}{2} \le \langle\lambda,\chi\rangle \leq \frac{\eta_\lambda}{2}\text{ for all } \lambda \colon \bG_m \to T\right. \right\}.
\]
\end{rem}

Let $w_0$ be the longest element of $W$.

\begin{lem} \label{lem:dominant}
Suppose that $\lambda$ and $\chi$ are dominant. Then for any $w \in W$,
\[
\pair{\lambda}{\chi} \ge \pair{w\lambda}{\chi} \ge \pair{w_0\lambda}{\chi}.
\]
\end{lem}

\begin{proof}
A reference for the first inequality is \cite[Corollary D.3]{SVdB}. For the second inequality, first note that $-w_0\chi$ is dominant. So applying the first inequality with $-w_0\chi$ in place of $\chi$, we conclude that $\pair{w\lambda}{w_0\chi} \ge \pair{\lambda}{w_0 \chi}$ for all $w \in W$. Since the pairing is $W$-invariant and $w_0^2 = 1$, this implies the second inequality for any $w \in W$.
\end{proof}

\begin{lem}\label{lem:polytopes}
Assume that $\varepsilon \in \op{Pic}(BG)_\bR$ is generic for $\ol{\Sigma}$. Then $M^+_\bR \cap ( -\rho + \frac{1}{2} \ol{\Sigma}_\varepsilon ) = M^+_\bR \cap \nabla_\varepsilon$. In particular, $\chi \in M^+ \cap (-\rho + \frac{1}{2} \ol{\Sigma}_\varepsilon)$ if and only if the character of $V(\chi)$ lies in $\nabla_\varepsilon$.
\end{lem}

\begin{proof}
If $\lambda$ is dominant, and hence $w_0 \lambda$ is anti-dominant, then 
\begin{align*}
\frac{\eta_\lambda}{2} &= \frac{1}{2} \langle \lambda, [(X^\dual)^{\lambda>0}] - [(\fg^\dual)^{\lambda>0}]\rangle = \frac{1}{2} \max \{\langle \lambda,\mu \rangle \mid \mu \in \ol{\Sigma} \} - \langle \lambda,\rho\rangle, \text{ and}\\
\frac{\eta_{w_0 \lambda}}{2} &= \frac{1}{2} \max \{\pair{w_0 \lambda}{\mu} \mid \mu \in \ol{\Sigma} \} + \pair{w_0 \lambda}{\rho}.
\end{align*}
Now suppose, in addition, that $\chi \in M^+_\bR \cap (-\rho + \frac{1}{2} \ol{\Sigma}_\varepsilon)$. If $\pair{\lambda}{\varepsilon} >0$, then
\begin{align*}
\pair{\lambda}{\chi} &\leq \frac{1}{2} \max \{ \pair{\lambda}{\mu} \mid \mu \in \ol{\Sigma} \} - \pair{\lambda}{\rho} = \frac{\eta_\lambda}{2}, \text{ and} \\
\pair{w_0 \lambda}{\chi} &> \frac{1}{2} \min \{ \pair{w_0 \lambda}{\mu} \mid \mu \in \ol{\Sigma} \} - \pair{w_0\lambda}{\rho}\\
&= - \frac{1}{2} \max \{\pair{w_0 \lambda}{\mu} \mid \mu \in \ol{\Sigma} \} + \pair{w_0\lambda}{\rho} = -\frac{\eta_{w_0\lambda}}{2},
\end{align*}
where we have used that $X$ is quasi-symmetric, and strict inequality on the second line uses that $\pair{w_0 \lambda}{\varepsilon} > 0$ and that $\chi \in -\rho + t \varepsilon + \frac{1}{2} \ol{\Sigma}$ for some $t>0$.

Since $\eta_{w\lambda} = \eta_\lambda$ for all $w \in W$, and every weight is in the $W$-orbit of a dominant weight, we can combine the previous inequalities with \autoref{lem:dominant} to conclude that $\frac{\eta_\lambda}{2} \ge \pair{\lambda}{\chi} > -\frac{\eta_{\lambda}}{2}$ for all weights $\lambda$ such that $\pair{\lambda}{\varepsilon} > 0$, which means that $\chi \in M^+_\bR \cap \nabla_\varepsilon$.

\medskip

Conversely, suppose that $\chi \in M^+_\bR \cap \nabla_\varepsilon$. By definition of $\nabla_\varepsilon$, there exists $t>0$ such that $\chi - t\varepsilon \in \nabla_\varepsilon$. Define
\[
r_0 = \min \{ r \ge 0 \mid \chi - t\varepsilon \in -\rho + r_0 \ol{\Sigma} \}.
\]
We will show that $\frac{1}{2} \ge r_0$. First, there exists $\lambda$ such that for all $\mu \in -\rho + r_0 \ol{\Sigma}$, we have $\pair{\lambda}{\chi - t\varepsilon} \ge \pair{\lambda}{\mu}$. Equivalently, $\pair{\lambda}{\chi - t\varepsilon + \rho} \ge \pair{\lambda}{\mu + \rho}$ for all such $\mu$. Since $\chi - t\varepsilon + \rho$ is dominant, if we choose $w \in W$ so that $w\lambda$ is dominant, then $\pair{w\lambda}{\chi - t\varepsilon + \rho} \ge \pair{\lambda}{\chi - t\varepsilon +\rho}$ (\autoref{lem:dominant}). Hence, replacing $\lambda$ by $w\lambda$ and using that $\ol{\Sigma}$ is $W$-invariant, we conclude that there exists a dominant weight $\lambda$ such that
\[
\pair{\lambda}{\chi - t\varepsilon +\rho} \ge r_0 \max \{\pair{\lambda}{\nu} \mid \nu \in \ol{\Sigma} \}.
\]
Furthermore, we may choose $\lambda$ to be a linear functional that is constant on a facet of $-\rho + r_0 \ol{\Sigma}$, and hence by \autoref{prop:semistable} we have $\langle \lambda, \varepsilon \rangle \ne 0$. 

Since $\eta_{-\lambda} = \eta_\lambda$ and $\chi - t\varepsilon \in \nabla_\varepsilon$, then we have, by definition of $\nabla_\varepsilon$, that $\frac{\eta_\lambda}{2} \ge \pair{\lambda}{\chi - t\varepsilon}$. Since $\lambda$ is dominant, then as above, we have $\frac{\eta_\lambda}{2} = \frac{1}{2} \max \{ \pair{\lambda}{\nu} \mid \nu \in \ol{\Sigma} \} - \pair{\lambda}{\rho}$. Combining this with the above inequality, we conclude that
\[
\frac{1}{2} \max \{ \pair{\lambda}{\nu} \mid \nu \in \ol{\Sigma} \} \ge \pair{\lambda}{\chi - t\varepsilon +\rho} \ge r_0 \max \{\pair{\lambda}{\nu} \mid \nu \in \ol{\Sigma} \}.
\]
Since $\varepsilon$ is in the linear span of $\ol{\Sigma}$ and $\pair{\lambda}{\varepsilon} \ne 0$, we know that $\pair{\lambda}{\nu} \ne 0$ for some $\nu \in \ol{\Sigma}$. Finally, by quasi-symmetry, $\nu \in \ol{\Sigma}$ implies that some negative multiple of $\nu$ is also in $\ol{\Sigma}$, so $\max\{\pair{\lambda}{\nu} \mid \nu \in \ol{\Sigma}\} > 0$. We conclude that $\frac{1}{2} \ge r_0$, as desired. This implies that $\chi \in -\rho + t\varepsilon + \frac{1}{2} \ol{\Sigma}$, and in particular, $\chi \in -\rho + \frac{1}{2} \ol{\Sigma}_{\varepsilon}$.
\end{proof}

\begin{cor} \label{cor:parallel-facets}
Assume that there exists $\varepsilon \in \op{Pic}(BG)_\bR$ which is generic for $\ol{\Sigma}$. Then each facet of $\ol{\nabla}$ is parallel to some facet of $\ol{\Sigma}$.
\end{cor}

\begin{proof}
Taking closures in \autoref{lem:polytopes} implies that $M_\bR^+ \cap \ol{\nabla} = M_\bR^+ \cap (-\rho + \frac{1}{2} \ol{\Sigma})$. In particular, this implies that $\ol{\nabla}$ is the $W$-orbit of $M_\bR^+ \cap (-\rho + \frac{1}{2} \ol{\Sigma})$. Each facet of $M_\bR^+ \cap (-\rho + \frac{1}{2} \ol{\Sigma})$ is either parallel to a facet of $\ol{\Sigma}$, or is contained in a facet of $M_\bR^+$. In particular, each facet of $\ol{\nabla}$ is a $W$-translate of one of these two types of facets. Since $\ol{\nabla}$ is $W$-invariant, if it contains any facet of $M_\bR^+$, then it contains all of them; however, we also know that $\ol{\nabla}$ is convex, so it cannot contain all of them. In particular, since $\ol{\Sigma}$ is also $W$-invariant, every facet of $\ol{\nabla}$ is parallel to some facet of $\ol{\Sigma}$.
\end{proof}

\section{The main theorem}

Given a quasi-symmetric representation $X$ of a reductive group $G$, our main theorem will identify the derived category of coherent sheaves on the GIT quotient ${\rm D}^b(X^{\rm{ss}}(\ell)/G)$ with a full subcategory of the equivariant derived category ${\rm D}^b(X/G)$, provided $\ell$ satisfies the genericity conditions discussed in the previous section.

\begin{defn}
For any region $\Omega \subset M_\bR$, let $\cM(\Omega) \subset {\rm D}^b(X/G)$ be the full subcategory which is (split) generated by objects of the form $\cO_X \otimes U$ for those $U \in \Rep(G)$ whose character is contained in $\Omega$.
\end{defn}

As above we shall denote by $\ol{\Sigma}$ the zonotope determined by the character of $X$.
\begin{thm} \label{thm:magic_windows}
Let $\delta, \ell \in \op{Pic}(BG)_\bR$ be characters such that $\partial(\delta + \ol{\nabla}) \cap M = \emptyset$. Then the restriction functor induces a fully faithful functor
\[
\op{res}_{X^{\rm{ss}}(\ell)} \colon \cM(\delta + \ol{\nabla}) \to {\rm D}^b(X^{\rm{ss}}(\ell)/G).
\]
If $(M_\bR^W)_{\ol{\Sigma}\text{-{\rm gen}}} \neq \emptyset$, then this is an equivalence whenever $X^{\rm{ss}}(\ell) = X^{\rm{s}}(\ell)$.\footnote{Note that if $X^{\rm{ss}}(\ell) = X^{\rm{s}}(\ell)$, then $X^{\rm{ss}}(\ell)/G$ is Deligne--Mumford. It follows that $X$ has finite generic stabilizer under the action of $T$, and hence $\ol{\Sigma}$ spans $M_\bR$.}
\end{thm}

We prove this theorem in the remainder of the section.

\begin{lem} \label{lem:hyperplane}
There exists a locally finite periodic hyperplane arrangement in $M_\bR$, such that if $\delta$ is outside of the hyperplanes, then $\partial(\delta + \ol{\nabla}) \cap M = \emptyset$.
\end{lem}

\begin{proof}
Given a hyperplane $H$, we have $(\delta + H) \cap M \ne \emptyset$ if and only if $\delta \in M-H$. So if $H_1, \dots, H_n$ are the hyperplanes such that $H_i \cap \ol{\nabla}$ is a facet, then we take our hyperplane arrangement to be the set of $m - H_i$ where $m \in M$. This is clearly periodic. Finally, since $M$ is a lattice, there exists $\varepsilon > 0$ such that any two points are at distance at least $\varepsilon$. So our arrangement is also locally finite.
\end{proof}

\begin{rem} \label{rem:alternate_main_thm}
Assuming the existence of an $\varepsilon \in (M_\bR^W)_{\ol{\Sigma}\text{-{\rm gen}}}$, the claim of \autoref{thm:magic_windows} is equivalent to the claim that for any $\delta \in M_\bR^W$, the restriction functor $\cM(\delta + \nabla_\varepsilon) \to {\rm D}^b(X^{\rm{ss}}(\ell)/G)$ is fully faithful and is an equivalence when $X^{\rm{ss}}(\ell)=X^{\rm{s}}(\ell)$.
\end{rem}

\begin{proof} [Proof of Remark]
Note that for any region $\Omega \subset M_\bR$, the category $\cM(\Omega)$ only depends on the intersection $\Omega \cap M$. Now if $\partial(\delta + \ol{\nabla}) \cap M = \emptyset$, then $(\delta + \ol{\nabla}) \cap M = (\delta + \nabla_\varepsilon) \cap M$, so $\cM(\delta+\ol{\nabla}) = \cM(\delta + \nabla_\varepsilon)$. Conversely, let $\delta$ be arbitrary and let $\varepsilon$ be generic for $\ol{\Sigma}$. Then $\varepsilon$ is generic for $\ol{\nabla}$, whose facets are parallel to those of $\ol{\Sigma}$ by \autoref{cor:parallel-facets}. In this case, $(\delta + \nabla_\varepsilon) \cap M = (\delta + r \varepsilon + \ol{\nabla}) \cap M$ for sufficiently small $r>0$, so $\cM(\delta + \nabla_\varepsilon) = \cM(\delta + r \varepsilon + \ol{\nabla})$ and $\delta + r\varepsilon$ satisfies the genericity hypothesis $\partial (\delta + r \varepsilon + \ol{\nabla}) \cap M = \emptyset$.
\end{proof}

\subsection{Magic windows}
\label{sect:magic_windows}

In geometric invariant theory, after fixing a linearization $\ell \in \op{Pic}(BG)_\bR$ and a Weyl-invariant inner product on $N$, one has a canonical sequence of locally closed subvarieties $Z_0^{\rm{ss}},\ldots,Z_n^{\rm{ss}} \subset X$ along with canonical one-parameter subgroups $\lambda_0, \ldots, \lambda_n$ such that $\lambda_i$ fixes $Z_i^{\rm{ss}}$ pointwise and is ``maximally destabilizing'' for those points.

In \cite[Theorem 2.10]{HL}, it is shown that for any choice of integers $w = (w_0,\ldots,w_n)$ the restriction functor $\cG^w \to {\rm D}^b(X^{\rm{ss}}(\ell)/G)$ is an equivalence, where by definition
\begin{equation} \label{eqn:Gw}
\cG^w := \left\{ F \in {\rm D}^b(X/G)\, \left|\, F|_{Z^{\rm{ss}}_i} \text{ has }\lambda_i \text{-weights in } [w_i,w_i+\eta_i) \text{ for all $i$} \right. \right\},
\end{equation}
and $\eta_i$ is the total $\lambda_i$-weight of $({\rm N}_{S_i}^\dual X)|_{Z_i}$. In fact, $\eta_i = \eta_{\lambda_i}$ in the notation of \S\ref{sect:zonotope}. Observe that the definition of $\cG^w$ makes sense as written for $w_i \in \bR$, and due to the half-open intervals of integer length in the definition of $\cG^w$ we have $\cG^w = \cG^{\lceil w \rceil}$ where $\lceil w \rceil = (\lceil w_0 \rceil,\ldots, \lceil w_n \rceil)$. It follows that \cite[Theorem 2.10]{HL} applies to $\cG^w$ for real $w$ as well.

\begin{lem} \label{lem:fully_faithful}
Let $\delta,\varepsilon \in M_\bR$ and let $\ell \in \op{Pic}(BG)_\bR = M_\bR^W$ be characters such that $\varepsilon$ is generic for $\ol{\nabla}$.\footnote{The same is true, with a slightly simpler proof, if instead we require $\varepsilon$ to be generic in the sense that $\pair{\lambda_i}{\varepsilon} \neq 0$ for all $i$.} Then there is a choice of $w= (w_0,\ldots,w_n)$ such that $\cM(\delta + \nabla_\varepsilon) \subset \cG^w$, and hence the restriction functor $\cM(\delta + \nabla_\varepsilon) \to {\rm D}^b(X^{\rm{ss}}(\ell)/G)$ is fully faithful.
\end{lem}

\begin{proof}
The vector bundle $\cO_X \otimes U|_{Z^{\rm{ss}}_i}$ lies in the relevant weight windows with respect to $\lambda_i$ as long as the $\lambda_i$-weights of $U$ lie in the interval $[w_i ,w_i+\eta_{\lambda_i})$. The width of this interval is precisely the width of $\ol{\nabla}$ in the $\lambda_i$-codirection, so we know that $\chi \in \delta + \ol{\nabla}$ implies that $\pair{\lambda_i}{\chi} \in \pair{\lambda_i}{\delta} + [-\eta_i/2,\eta_i/2]$. We must choose $w_i$ so that the previous claim is still true after replacing $\ol{\nabla}$ with $\nabla_\varepsilon$ and the closed interval $[-\eta_i/2,\eta_i/2]$ with the half-open interval $[-\eta_i/2,\eta_i/2)$.

We choose a very small $0<a\ll 1$ and let
\[
w_i = \pair{\lambda_i}{\delta} - \frac{\eta_{\lambda_i}}{2} + 
\begin{cases} a & \text{if } \pair{\lambda_i}{\varepsilon}>0 \\ 0 & \text{else} \end{cases}.
\]
By definition $\chi \in \delta + \nabla_\varepsilon$ if and only if $\chi - t \varepsilon \in \delta + \ol{\nabla}$ for all $0\leq t\ll 1$. This implies that $\pair{\lambda_i}{\chi} \in t\pair{\lambda_i}{\varepsilon} +  \pair{\lambda_i}{\delta} + [-\eta_i/2,\eta_i/2]$ for all $0 \leq t \ll 1$, and so 
\begin{equation} \label{eqn:weight_bounds}
\pair{\lambda_i}{\chi} \in 
\begin{cases}  \pair{\lambda_i}{\delta} + (-\eta_{\lambda_i}/2,\eta_{\lambda_i}/2] & \text{if } \pair{\lambda_i}{\varepsilon} > 0 \\ \pair{\lambda_i}{\delta} + [-\eta_{\lambda_i}/2,\eta_{\lambda_i}/2) & \text{if } \pair{\lambda_i}{\varepsilon} < 0 \end{cases}.
\end{equation}
In both cases, if we assume that $\chi \in M$, then because $a$ is very small we will have $\chi \in [w_i,w_i+\eta_i)$. So in particular as long as $\pair{\lambda_i}{\varepsilon} \neq 0$, then for any $U$ whose character lies in $\delta + \nabla_\varepsilon$, the locally free sheaf $\cO_X \otimes U$ will satisfy the grade restriction rule for $\cG^w$ with respect to $\lambda_i$.

What remains is the case where $\pair{\lambda_i}{\varepsilon} = 0$ for some $i$. In this case we claim that for any $\chi \in \delta + \nabla_\varepsilon$ we have $\pair{\lambda_i}{\chi} \in (w_i,w_i+\eta_i)$, because any $\chi \in \delta + \ol{\nabla}$ such that $\pair{\lambda_i}{\chi} = w_i$ or $w_i+\eta_{\lambda_i}$ cannot lie in $\nabla_\varepsilon$. If it did, then $\chi - t\varepsilon \in \delta + \ol{\nabla}$ would also maximize (respectively, minimize) $\pair{\lambda_i}{-}$ for all $0 \leq t \ll 1$. Every maximizer (respectively, minimizer) of this function must occur on the boundary $\partial(\delta + \ol{\nabla})$, so we conclude that the line segment $\chi -t \varepsilon$ must be contained in the boundary, which contradicts that $\varepsilon$ is generic for $\ol{\nabla}$.
\end{proof}

\subsection{Constructing complexes in ${\rm D}^b(X/G)$.}

To set notation, we recall the Borel--Weil--Bott theorem. For a reference, see \cite[II, Cor. 5.5, 5.6]{jantzen}. The convention for Borel subgroups in \cite[II, 1.8]{jantzen} matches the one in \S\ref{sec:notation}. Given a weight $\mu$ of $T$ and $w \in W$, define 
\[
w * \mu = w(\mu + \rho) - \rho.
\]
If $\mu + \rho$ has a trivial stabilizer in $W$, then there is a unique $w \in W$ such that $w*\mu$ is a dominant weight, and in that case, we write $\mu^+ = w*\mu$. As before, $w_0 \in W$ is the longest element. A character $\chi$ of $T$ determines a $1$-dimensional representation $k_\chi$ of $B$, and hence a $G$-equivariant line bundle $\cL(\chi) = G \times_B k_\chi$ on $G/B$.

\begin{recollection}[Borel--Weil--Bott]
If $\mu+\rho$ has a nontrivial stabilizer in $W$, then all cohomology groups ${\rm H}^i(G/B; \cL(\mu))$ vanish. Otherwise, let $w$ be such that $w*\mu$ is dominant. Then the cohomology of $\cL(\mu)$ vanishes except in degree $\ell(w)$, and we have a $G$-equivariant isomorphism ${\rm H}^{\ell(w)}(G/B; \cL(\mu)) \cong V(\mu^+)$.
\end{recollection}

\begin{lem} \label{lem:BWB-dual}
For any weight $\alpha$ such that $\alpha^+$ is defined, then $(-w_0 \alpha)^+$ is defined, and
\[
V((-w_0\alpha)^+)^\vee \cong V(\alpha^+).
\]
\end{lem}

\begin{proof}
Pick $u \in W$ such that $\alpha^+ = u(\alpha + \rho) - \rho$. Let $u' = w_0 u w_0^{-1}$. Since $-w_0 \rho = \rho$, we have
\begin{align*}
u'(-w_0 \alpha + \rho) - \rho &= -u'w_0(\alpha + \rho) - \rho
= -w_0 u(\alpha + \rho) - \rho\\
&= -w_0( u (\alpha + \rho) - \rho) = -w_0 (\alpha^+).
\end{align*}
Since $\alpha^+$ is dominant, the same is true for $-w_0(\alpha^+)$, so $(-w_0 \alpha)^+$ exists and is equal to $-w_0(\alpha^+)$. We finish using the fact that for dominant $\mu$, we have $V(\mu)^\vee \cong V(-w_0 \mu)$. 
\end{proof}

Let $\lambda$ be an anti-dominant one-parameter subgroup. Define a subspace $X^{\lambda \geq 0} \subset X$ spanned by eigenvectors of nonnegative weight with respect to the action of $\bG_m$ via $\lambda$. Then $X^{\lambda \ge 0}$ is a $B$-submodule of $X$. Define $\cS(\lambda) = G \times_{B} X^{\lambda \ge 0}$, which is a $G$-equivariant vector bundle on $G/B$; further, it is a subbundle of the trivial bundle $\cO_{G/B} \times X$. 

Let $\xi(\lambda)$ be the locally free subsheaf of $\cO_{G/B} \otimes X^\dual$ which is the annihilator of $\cS(\lambda)$. These are the local linear equations that define $\cS(\lambda)$ in $G/B \times X$, and we have a locally free resolution of $\cO_{\cS(\lambda)}$ over $G/B \times X$ given by a Koszul complex $\bigwedge^\bullet p^*\xi(\lambda)$, where $p \colon G/B \times X \to G/B$ is the projection onto the first factor. Let $\pi \colon G/B \times X \to X$ denote the projection onto the second factor. By \cite[Theorem 5.1.2, Prop. 5.2.5]{weyman}, the derived pushforward ${\rm R} \pi_*(\cO_{\cS(\lambda)} \otimes \cL(\chi))$ is quasi-isomorphic to a minimal (i.e., the entries in the differentials vanish at the origin of $X$) complex $(C_{\lambda, \chi})_\bullet$ with terms 
\[
(C_{\lambda, \chi})_i = \cO_X \otimes \bigoplus_{j \in \bZ} {\rm H}^j(G/B; \cL(\chi) \otimes \bigwedge^{i+j} \xi(\lambda)).
\]

\begin{prop} \label{prop:C_unstable}
\begin{enumerate}[\rm (a)]
\item If $\langle \lambda, \ell \rangle > 0$, then the homology of $C_{\lambda, \chi}$ is supported on $X \setminus X^{\rm ss}(\ell)$.
\item The terms of the complex $C_{\lambda, \chi}$ are direct sums of locally free sheaves of the form
\[
\cO_X \otimes V((\chi - \beta_{i_1} - \cdots - \beta_{i_p})^+)
\]
where $i_1, \dots, i_p$ are distinct and $\langle \lambda, \beta_{i_j} \rangle < 0$. If $\chi$ is dominant and $(\chi - \beta_{i_1} - \cdots - \beta_{i_p})^+ = \chi$ implies $p=0$, then $\cO_X \otimes V(\chi)$ appears exactly once.
\end{enumerate}
\end{prop}

\begin{proof}
(a) The image of $\cS(\lambda)$ under $\pi$ is contained in the unstable locus by the Hilbert--Mumford criterion.

(b) First, $\bigwedge^\bullet \xi(\lambda)$ has a $G$-equivariant filtration by the line bundles $\cL(-\beta_{i_1} - \cdots - \beta_{i_p})$ where $i_1, \dots, i_p$ are distinct, and $\langle \lambda, \beta_{i_j} \rangle < 0$. This filtration gives a $G$-equivariant spectral sequence to compute the terms of $C_{\lambda, \chi}$ whose second page can be computed using Borel--Weil--Bott. The assumption of the last sentence implies that $\cO_X \otimes V(\chi)$ survives to the infinity page.
\end{proof}

Note that $-w_0\lambda$ is also anti-dominant. The derived pushforward ${\rm R}\pi_* ( \cO_{\cS(-w_0\lambda)} \otimes \cL(-w_0 \chi))$ is quasi-isomorphic to a minimal complex $(D_{\lambda, \chi})_\bullet$ with terms
\begin{align*}
(D_{\lambda, \chi})_i &= \cO_X \otimes \bigoplus_{j \in \bZ} {\rm H}^j(G/B; \cL(-w_0 \chi) \otimes \bigwedge^{i+j} \xi(-w_0\lambda))
\end{align*}

\begin{prop} \label{prop:D_unstable}
\begin{enumerate}[\rm (a)]
\item If $\langle \lambda, \ell \rangle < 0$, then the homology of $D_{\lambda, \chi}^\dual$ is supported on $X \setminus X^{\rm ss}(\ell)$.

\item The terms of the complex $D_{\lambda, \chi}^\dual$ are direct sums of locally free sheaves of the form
\[
\cO_X \otimes V((\chi + \beta_{i_1} + \cdots + \beta_{i_p})^+)
\]
where $i_1, \dots, i_p$ are distinct and $\pair{\lambda}{\beta_{i_j}} > 0$. If $\chi$ is dominant and $(\chi + \beta_{i_1} + \cdots + \beta_{i_p})^+ = \chi$ implies $p=0$, then $\cO_X \otimes V(\chi)$ appears exactly once.
\end{enumerate}
\end{prop}

\begin{proof}
(a) If $\pair{\lambda}{\ell} < 0$, then $\pair{-w_0\lambda}{\ell} > 0$, so the image of $\cS(-w_0\lambda)$ under $\pi$ is contained in the unstable locus by the Hilbert--Mumford criterion.

(b) First, $\cL(-w_0\chi) \otimes \bigwedge^\bullet \xi(-w_0\lambda)$ has a $G$-equivariant filtration by the line bundles $\cL(-w_0\chi -\beta_{i_1} - \cdots - \beta_{i_p})$ where $i_1, \dots, i_p$ are distinct, and $\pair{-w_0 \lambda}{\beta_{i_j}} < 0$. Equivalently, by replacing $\beta_i$ with $w_0 \beta_i$, we can rewrite them as the line bundles $\cL(-w_0(\chi + \beta_{i_1} + \cdots + \beta_{i_p}))$ where $i_1, \dots, i_p$ are distinct and $\pair{\lambda}{\beta_{i_j}} > 0$. By \autoref{lem:BWB-dual}, we have $V((-w_0(\chi + \beta_{i_1} + \cdots + \beta_{i_p}))^+)^\vee \cong V((\chi + \beta_{i_1} + \cdots + \beta_{i_p})^+)$. The rest of the proof is similar to the proof for \autoref{prop:C_unstable}.
\end{proof}

\begin{rem} \label{rmk:euler-char}
It is easy to give a formula for the equivariant Euler characteristics of the complexes we just constructed, and we will use this formula in \autoref{rmk:K-theory} to study the effect of our derived equivalences on $K$-theory. First, the bundle $\xi(\lambda)$ has a $G$-equivariant filtration by the line bundles $\cL(-\beta_{i_1} - \cdots - \beta_{i_p})$ where the $i_1,\dots,i_p$ are distinct and $\pair{\lambda}{\beta_{i_j}} < 0$. For any weight $\nu$, define 
\[
{\rm ch}(\nu) = \frac{\sum_{w \in W} (-1)^{\ell(w)} \epsilon_{w(\nu+\rho)}}{ \sum_{w \in W} (-1)^{\ell(w)} \epsilon_{w \rho}}
\]
where $\epsilon_\alpha$ is the character of the weight $\alpha$ for the torus $T$. By the Weyl denominator formula and Borel--Weil--Bott theorem, this is the Euler characteristic $\sum_j (-1)^j [{\rm H}^j(G/B; \cL(\nu))]$. Let $\cI_-(\lambda) = \{i \mid \pair{\lambda}{\beta_i}<0\}$. Then the Euler characteristic of $C_{\lambda, \chi}$ is 
\[
\sum_{S \subseteq \cI_-(\lambda)} (-1)^{|S|} {\rm ch}\bigg(\chi - \sum_{i \in S}\beta_i\bigg).
\]
Similarly, setting $\cI_+(\lambda) = \{ i \mid \pair{\lambda}{\beta_i}>0\}$, the Euler characteristic of $D_{\lambda,\chi}^\vee$ is 
\[
\sum_{S \subseteq \cI_+(\lambda)} (-1)^{|S|} {\rm ch}\bigg(\chi + \sum_{i \in S}\beta_i\bigg).
\]
\end{rem}

\subsection{Proof of \autoref{thm:magic_windows}: the {\v S}penko--Van den Bergh algorithm revisited} \label{sec:SVDB}

In this section we complete the proof of \autoref{thm:magic_windows}. The fully faithfulness part of the theorem is covered by \autoref{lem:fully_faithful} above combined with the observation that $\cM(\delta + \ol{\nabla}) = \cM(\delta + \nabla_\varepsilon)$ for any $\varepsilon \in M_\bR$ when $\partial(\delta + \ol{\nabla}) \cap M = \emptyset$. It therefore suffices to show that the vector bundles in $\cM(\delta+\ol{\nabla})$ generate ${\rm D}^b(X^{\rm{ss}}(\ell)/G)$ when $X^{\rm{ss}}(\ell) = X^{\rm{s}}(\ell)$. Note that by \autoref{cor:perturbation} we may perturb $\ell$ slightly so that it is generic for $\ol{\Sigma}$ without changing $X^{\rm{ss}}(\ell)$, so we will assume for the remainder of the proof that this is the case. The argument closely follows the algorithm used to produce non-commutative resolutions of $\Spec(\cO_X^G)$ in \cite{SVdB}, with only minor modifications required because our goal, generation over the semistable locus, is different from that of \cite{SVdB}, which sought categories generated by equivariant vector bundles which have finite global dimension.

First, we reduce to the case where $G$ is connected. Let $G_e$ be the connected component of the identity of $G$. It follows from the Hilbert--Mumford criterion that $X^{G\text{\rm{-ss}}}(\ell) = X^{G_e\text{\rm{-ss}}}(\ell)$, and the polytope $\nabla_\varepsilon$ defined in terms of $G$ agrees with that defined in terms of $G_e$.\footnote{The maximal torus of $G$ is also a maximal torus of $G_e$, so there is no ambiguity in the meaning of $M$.} The map of stacks $\pi \colon X^{\rm{ss}}(\ell) / G_e \to X^{\rm{ss}}(\ell)/G$ is a representable finite \'{e}tale morphism with fiber $G/G_e$. In particular, if a set of vector bundles $\{\cV_\alpha\}$ generates ${\rm D}^b(X^{\rm{ss}}(\ell)/G_e)$, then the bundles $\pi_\ast \cV_\alpha = \cV_\alpha \otimes k[G/G_e]$ with their evident $G$-equivariant structure generate ${\rm D}^b(X^{\rm{ss}}(\ell)/G)$. Furthermore, the character of an irreducible representation $U \in \Irr(G)$ lies in $\delta + \nabla_\varepsilon$ if and only if the character of $U \otimes k[G/G_e]$ does, because $T$ acts trivially on $k[G/G_e]$.

We can therefore reduce the essential surjectivity part of \autoref{thm:magic_windows} to the case when $G = G_e$. In fact we show that $\cM(\delta + \ol{\nabla}) \to {\rm D}^b(X^{\rm{ss}}(\ell)/G)$ is essentially surjective for any $\delta$, which follows from \autoref{lem:polytopes} combined with the following:

\begin{prop}
Let $G$ be connected. Assume that $X$ is quasi-symmetric and that the stack $X^{T\text{\rm -ss}}(\ell)/T$ is Deligne--Mumford. Then ${\rm D}^b(X^{\rm{ss}}(\ell)/G)$ is generated by $\cO_X \otimes V(\mu)$ where $\mu \in M^+ \cap (-\rho + \delta + \frac{1}{2} \ol{\Sigma})$.
\end{prop}

\begin{proof}
For any $\chi \in M^+$, we want to show that $\cO_X \otimes V(\chi)$ lies in the full triangulated subcategory generated by $\cO_X \otimes V(\mu)$ with $\mu \in M^+ \cap (-\rho + \delta + \frac{1}{2} \ol{\Sigma})$. We do this by a double induction first on the number
\[
r_\chi := \min \{ r \geq 0 \mid \chi \in -\rho + r \ol{\Sigma} + \delta \},
\]
and then with respect to the integer $p_\chi$, which we define to be the minimal number of $a_i$ which are equal to $-r_\chi$ among all ways of writing $\chi = -\rho + \sum_i a_i \beta_i + \delta$ with $a_i \in [-r_\chi, 0]$. Note that $r_\chi$ is a real number in general, but the set of possible $r_\chi$ is discrete if we restrict to $\chi \in M^+$. If $r_\chi \le \frac{1}{2}$, then $\chi \in -\rho + \delta + \frac{1}{2} \ol{\Sigma}$, so there is nothing to show. So we assume now that $r_\chi > \frac{1}{2}$.

First, there exists $\lambda$ such that for all $\mu \in \delta + r_\chi \ol{\Sigma}$, we have $\langle \lambda, \chi + \rho \rangle \ge \langle \lambda, \mu + \rho \rangle$. We may choose $\lambda$ to be a linear functional that is constant on a facet of $-\rho + r_\chi \ol{\Sigma} + \delta$, and so by \autoref{prop:semistable}, we have $\langle \lambda, \ell \rangle \ne 0$.  Let $w \in W$ be such that $w\lambda$ is dominant. By \autoref{lem:dominant}, $\pair{w\lambda}{\chi + \rho} \ge \pair{\lambda}{\chi+\rho}$. In particular, $\pair{w\lambda}{\chi + \rho} \ge \pair{w\lambda}{w(\mu+\rho)}$ for all $\mu \in \delta + r_\chi \ol{\Sigma}$. So, replacing $\lambda$ by $-w\lambda$, we may assume that $\lambda$ is anti-dominant, that $\langle \lambda, \chi \rangle \le \langle \lambda, \mu \rangle$ for all $\mu \in -\rho + r_\chi \ol{\Sigma} + \delta$, and that $\langle \lambda, \ell \rangle \ne 0$ (here we use that $\delta + r_\chi \ol{\Sigma}$ is $W$-invariant).

\begin{lem} \label{lem:rp-chi}
Write $\chi= -\rho + \sum_i a_i \beta_i + \delta$ with $a_i \in [-r_\chi, 0]$.
  \begin{enumerate}[\rm (a)]
  \item $r_\chi$ and $p_\chi$ depend only on the $W$-orbit of $\chi$ for the $*$-action.
  \item If $\langle \lambda, \beta_i \rangle > 0$, then $a_i = -r_\chi$.
  \item If $\langle \lambda, \beta_i \rangle < 0$, then $a_i = 0$.
  \end{enumerate}
\end{lem}

\begin{proof}
(a) This is clear from the definitions.

(b) Assume that the statement is false, i.e., there exists $i$ such that $\langle \lambda, \beta_i \rangle > 0$ and $0 \ge a_i > -r_\chi$. Then there exists $t > 0$ such that $\chi - t \beta_i \in -\rho + r_\chi \ol{\Sigma} + \delta$ and
\[
\langle \lambda, \chi - t \beta_i \rangle = \langle \lambda, \chi \rangle - t \langle \lambda, \beta_i \rangle < \langle \lambda, \chi \rangle,
\]
which contradicts the property for which $\lambda$ was chosen.

(c) This is similar to (b). 
\end{proof}

We now show that there exists a complex of free $\cO_X$-modules whose restriction to the semistable locus is exact, which contains the term $\cO_X \otimes V(\chi)$, and all other terms are of the form $\cO_X \otimes V(\mu)$ with either $r_\mu < r_\chi$ or $r_\mu=r_\chi$ and $p_\mu < p_\chi$. In particular, $\mu \ne \chi$ for all other terms, so once the existence of such a complex is established, we can conclude, by induction, that $\cO_X \otimes V(\chi)$ is generated by $\cO_X \otimes V(\alpha)$ with $\alpha \in M^+ \cap (-\rho + \delta + \frac{1}{2} \ol{\Sigma})$.

\medskip

\noindent {\bf Case 1:} $\langle \lambda, \ell \rangle < 0$. Consider the complex $D^\vee_{\lambda,\chi}$. By \autoref{prop:D_unstable}, the homology of $D^\vee_{\lambda, \chi}$ is supported in the unstable locus, and hence the restriction $D^\vee_{\lambda, \chi}|_{X^{\rm{ss}}(\ell)}$ is exact. Also, every term in $D^\vee_{\lambda, \chi}$ is a direct sum of modules of the form $\cO_X \otimes V((\chi + \beta_{i_1} + \cdots + \beta_{i_p})^+)$ with $\langle \lambda, \beta_{i_j} \rangle > 0$ and the $i_1, \dots, i_p$ are distinct.

We claim that if $\mu = (\chi + \beta_{i_1} + \cdots + \beta_{i_p})^+$ where $p>0$, then either $r_\mu < r_\chi$ or else $r_\mu=r_\chi$ and $p_\mu < p_\chi$. Using \autoref{lem:rp-chi}(a), we may replace $\mu$ with $\mu' = \chi + \beta_{i_1} + \cdots + \beta_{i_p}$. At this point, one can proceed as in \cite[\S 12.1]{SVdB}, so we omit the details, but see \autoref{fig:sym3} for a small example.

\begin{figure} 
\begin{tabular}{ll}
Case 1: use $D_{\lambda, \chi}^\vee$ & Case 2: use $C_{\lambda, \chi}$ \\
\begin{tikzpicture}[scale=.45]
\filldraw[black!30!white] (-4.5,0) -- (-3,3) -- (0,4.5) -- (4.5,4.5) -- (4.5,0) -- (3,-3) -- (0,-4.5) -- (-4.5,-4.5) -- cycle;
\filldraw[black!10!white] (-3,0) -- (-2,2) -- (0,3) -- (3,3) -- (3,0) -- (2,-2) -- (0,-3) -- (-3,-3) -- cycle;
\draw[step=1, color=gray] (-6,-5) grid (6,5);
\node[below] at (-1.5,4) {$r_\chi \ol{\Sigma}$};
\node[below] at (-1.5,2.25) {$\frac{1}{2}\ol{\Sigma}$};
\draw[<->, thick] (-6,0) -- (6,0);
\draw[<->, thick] (0,-5) -- (0,5);
\draw[dashed, thick] (2,-5) -- (6,3);
\draw[->, thick] (6,3) -- (5,3.5);
\node[above] at (5.5,3.25) {$-\lambda_2$};
\foreach \i in {(1,2), (4,2), (-1,1), (2,1), (1,-1), (4,-1), (-1,-2), (2,-2)}
{ \filldraw[red] \i circle (3pt); }
\node[red] at (4,-1) {$\times$};
\end{tikzpicture}
&
\begin{tikzpicture}[scale=.45]
\filldraw[black!30!white] (-4,0) -- (-8/3,8/3) -- (0,4) -- (4,4) -- (4,0) -- (8/3,-8/3) -- (0,-4) -- (-4,-4) -- cycle;
\filldraw[black!10!white] (-3,0) -- (-2,2) -- (0,3) -- (3,3) -- (3,0) -- (2,-2) -- (0,-3) -- (-3,-3) -- cycle;
\draw[step=1, color=gray] (-6,-5) grid (6,5);
\node[below] at (-1.5,3.75) {$r_\chi \ol{\Sigma}$};
\node[below] at (-1.5,2) {$\frac{1}{2}\ol{\Sigma}$};
\draw[<->, thick] (-6,0) -- (6,0);
\draw[<->, thick] (0,-5) -- (0,5);
\draw[dashed, thick] (-6,-4) -- (6,-4);
\draw[->, thick] (4.5,-4) -- (4.5,-3);
\node[right] at (4.5,-3.5) {$\lambda_1^\sigma$};
\foreach \i in {(2,2), (1,0), (0,1), (-1,-1), (2,-1), (0,-2), (1,-3), (-1,-4)}
{ \filldraw[red] \i circle (3pt); }
\node[red] at (-1,-4) {$\times$};
\end{tikzpicture}
\end{tabular}
\caption{\footnotesize Here we take $X = T^* {\rm Sym}^3(k^2)$ under the action of ${\bf GL}_2(k)$ with $\ell = (1,1)$. The maximal destabilizing one-parameter subgroups in $N$ are $\lambda_0 = (1,1)$, $\lambda_1 = (1,0)$, $\lambda_2 = (2,-1)$. We let $\sigma$ denote the nontrivial element of the Weyl group, so that $\lambda_i^\sigma$ also define facets of $\ol{\Sigma}$. In both cases, $\chi$ is represented by a red $\times$, and we attempt to find an acyclic complex relating $\cO_X \otimes V(\chi)$ to representations with strictly smaller $r_\chi$ and $p_\chi$. 
\emph{Left side:} We have $r_\chi = 3/4$ and $\lambda = -\lambda_2$, so $\pair{\lambda}{\ell}<0$ and the red dots show those $\mu$ such that $\cO_X \otimes V(\mu^+)$ might appear in $D^\dual_{\lambda,\chi}$. 
\emph{Right side:} We have $r_\chi = 2/3$ and $\lambda = \lambda_1^\sigma$, so $\pair{\lambda}{\ell} > 0$ and the red dots represent the weights $\mu$ for which $\cO_X \otimes V(\mu^+)$ might appear in $C_{\lambda,\chi}$. } \label{fig:sym3}
\end{figure}
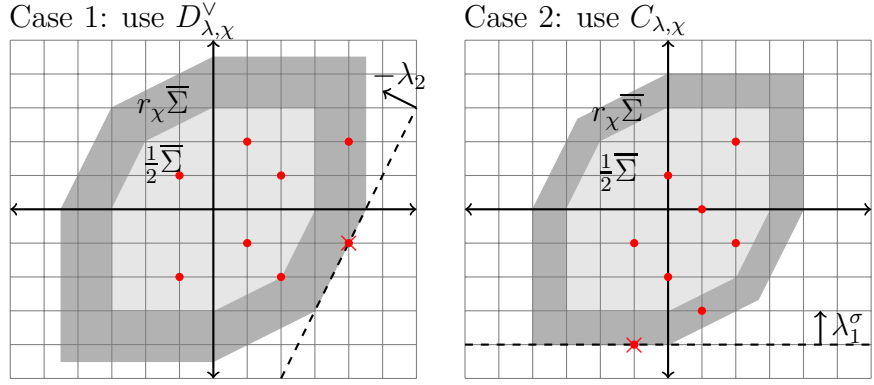

\medskip

\noindent {\bf Case 2:} $\langle \lambda, \ell \rangle > 0$. Consider the complex $C_{\lambda, \chi}$. By \autoref{prop:C_unstable}, the homology of $C_{\lambda, \chi}$ is supported in the unstable locus, and hence the restriction $C_{\lambda, \chi}|_{X^{\rm{ss}}(\ell)}$ is exact. Also, every term in $C_{\lambda, \chi}$ is a direct sum of modules of the form $\cO_X \otimes V((\chi - \beta_{i_1} - \cdots - \beta_{i_p})^+)$ with $\langle \lambda, \beta_{i_j} \rangle < 0$ and the $i_1, \dots, i_p$ are distinct. 

We claim that if $\mu = (\chi - \beta_{i_1} - \cdots - \beta_{i_p})^+$ where $p>0$, then either $r_\mu < r_\chi$ or else $r_\mu=r_\chi$ and $p_\mu < p_\chi$. Using \autoref{lem:rp-chi}(a), we may replace $\mu$ with $\mu' = \chi - \beta_{i_1} - \cdots - \beta_{i_p}$. The idea is similar to \cite[\S 12.1]{SVdB}, but we have to modify certain points. For clarity, we provide the full details of the proof of the claim.

Write $\chi = -\rho + \sum_i a_i \beta_i + \delta$ where exactly $p_\chi$ of the $a_i$ are equal to $-r_\chi$. We have
\[
\mu'=-\rho+\sum_i a'_i\beta_i + \delta,
\qquad \qquad a'_i=
\begin{cases}
a_i & \text{if } i\notin \{i_1, \dots, i_p\}\\
a_i-1 & \text{if }i \in \{i_1, \dots, i_p\}
\end{cases}.
\]
As in \cite[\S 12.1]{SVdB}, the main idea is to redistribute the coefficients of this expression in such a way that it is manifestly clear that either $r_{\mu'} < r_\chi$ or that $r_{\mu'}=r_\chi$ and $p_{\mu'} < p_\chi$. In order to do this and take advantage of the quasi-symmetric assumption on $X$, we rewrite this expression as
\[
\mu'= -\rho+\sum_L(\sum_{\beta_i\in L \setminus 0} a_i'\beta_i) +\delta
\]
where the sum is over all lines $L \subset M_\bR$ through the origin. We now consider each expression $\sum_{\beta_i \in L} a'_i \beta_i$ for a fixed line $L$. Define \begin{align*}
S^L &= \{i_1, \dots, i_p \} \cap \{i \mid \beta_i \in L\},\\
T^L &= \{i \mid \langle \lambda, \beta_i \rangle > 0,\ \beta_i \in L\},\\
U^L &= \{i \mid \langle \lambda, \beta_i \rangle < 0,\ \beta_i \in L\}.
\end{align*}

If $S^L = \emptyset$ then $a'_i=a_i$ for all $i$ such that $\beta_i\in L$ and hence $0 \ge a'_i \ge -r_\chi$, so we will leave the sum $\sum_{\beta_i \in L} a'_i \beta_i$ alone. For the remainder of the proof, we consider the case that $S^L \neq\emptyset$. In particular, $\lambda$ does not vanish on $L$. Let $\gamma$ be the unique vector on $L$ defined by $\langle \lambda, \gamma \rangle = 1$. For $\beta_i\in L$, we have $\beta_i= \langle \lambda, \beta_i \rangle \gamma$. Define two numbers
\begin{align*}
\alpha &= -\sum_{i\in S^L} \langle \lambda, \beta_i \rangle - \sum_{j\in T^L} r_{\chi}\langle \lambda, \beta_j \rangle, \qquad c =\sum_{i\in T^L} \langle \lambda, \beta_i \rangle.
\end{align*}

\begin{lem} \label{lem:alpha-ineq}
\begin{enumerate}[\rm (a)]
\item $c>0$.
\item $r_\chi > \frac{\alpha}{c} > -r_\chi$.
\item $\displaystyle \sum_{\beta_i\in L} a'_i\beta_i = \sum_{i\in T^L} \frac{\alpha}{c} \beta_i = \sum_{i\in U^L} -\frac{\alpha}{c} \beta_i$.
\end{enumerate}
\end{lem}

\begin{proof}
(a) $\langle \lambda, \beta_i \rangle > 0$ if $i \in T^L$ and $T^L \ne \emptyset$ (since $S^L \ne \emptyset$ and using quasi-symmetry).

(b) Since $c>0$, it suffices to show that $r_\chi c > \alpha > -r_\chi c$. For the inequality $\alpha > -r_\chi c$, use that $-\langle \lambda, \beta_i \rangle > 0$ for $i\in S^L$, and that $S^L \ne \emptyset$. For the inequality $r_\chi c > \alpha$, we have
\begin{align*}
r_\chi c = \sum_{i \in T^L} r_\chi \langle \lambda, \beta_i \rangle > \sum_{i \in T^L} (1-r_\chi) \langle \lambda, \beta_i \rangle = -\sum_{i \in U^L} \langle \lambda, \beta_i \rangle - \sum_{j \in T^L} r_\chi \langle \lambda, \beta_j \rangle \ge \alpha.
\end{align*}
In the first inequality, we used that $r_\chi > \frac{1}{2}$ implies that $r_\chi > 1-r_\chi$. The second equality uses the quasi-symmetric condition, which translates to $-\sum_{i \in U^L} \langle \lambda, \beta_i \rangle = \sum_{j \in T^L} \langle \lambda, \beta_j \rangle$. The last inequality uses that $S^L \subseteq U^L$ and that $-\langle \lambda, \beta_i \rangle > 0$ for $i \in U^L$. 

(c) By \autoref{lem:rp-chi}, we have $a_i =  0$ if $i \in U^L$ and $a_i = -r_\chi$ if $i \in T^L$. This implies that $\sum_{\beta_i \in L} a'_i \beta_i = \alpha \gamma$. Now continue using the definition of $\alpha$ and $c$:
\[
\alpha \gamma = \frac{\alpha}{c} c \gamma = \sum_{i \in T^L} \frac{\alpha}{c} \langle \lambda, \beta_i \rangle \gamma = -\sum_{i \in U^L} \frac{\alpha}{c} \beta_i.
\]
The equality $\sum_{i \in T^L} \beta_i = -\sum_{j \in U^L} \beta_j$ follows from the fact that $X$ is quasi-symmetric.
\end{proof}

Using \autoref{lem:alpha-ineq}, we can rewrite $\sum_{\beta_i \in L} a'_i \beta_i$ as a sum where the coefficients are in the half-open interval $(-r_\chi, 0]$ (which expression is used  depends on whether $\alpha \ge 0$ or $\alpha \le 0$). Doing all of these rewrites, we end up with an expression for $\mu'$ with coefficients in $[-r_\chi,0]$ which implies that $r_{\mu'} \le r_\chi$. If the inequality is strict, we're done. Otherwise, note that there is at least one line $L$ such that $S^L \ne \emptyset$. In such a line, we have removed all terms that have coefficient $-r_\chi$ (and there is at least one such term since $T^L \ne \emptyset$). In particular, $p_{\mu'} < p_\chi$.
\end{proof}

\section{Equivalences of derived categories, tilting bundles, and algebraic K-theory}
\label{sect:equivalences}

Throughout this section we fix a quasi-symmetric representation $X$ of a split reductive group $G$, and let $\ol{\Sigma} \subset M_\bR$ denote the zonotope associated to the character of $X$. The first consequence of \autoref{thm:magic_windows} is the following:

\begin{cor} \label{cor:equivalences}
If $(M_\bR^W)_{\ol{\Sigma}\text{-{\rm gen}}} \neq \emptyset$, then for any $\ell, \ell', \delta \in M_\bR^W$ such that $X^{\rm{ss}}(\ell) = X^{\rm{s}}(\ell)$, $X^{\rm{ss}}(\ell') = X^{\rm{s}}(\ell')$, and $\partial(\delta + \ol{\nabla}) \cap M = \emptyset$, one has an equivalence of derived categories
\[
F_{\ell, \ell', \delta} \colon \xymatrix{ {\rm D}^b(X^{\rm{ss}}(\ell) / G) \ar[rr]^-{\op{res}_{X^{\rm{ss}}(\ell)}^{-1}} & & \cM(\delta + \ol{\nabla}) \ar[rr]^-{\op{res}_{X^{\rm{ss}}(\ell')}} & & {\rm D}^b(X^{\rm{ss}}(\ell')/G)}.
\]
which is linear over $\cO_X^G$ and restricts to the identity morphism over the preimage of the open subset of stable points $X^{s}(0) \git G \subset \Spec(\cO_X^G)$.
\end{cor}

\begin{proof}
The fact that $F_{\ell,\ell',\delta}$ is an $\cO_X^G$-linear functor follows from the fact that the restriction functor $\op{res}_{X^{\rm{ss}}(\ell)} \colon \cM(\delta+\ol{\nabla}) \to {\rm D}^b(X^{\rm{ss}}(\ell)/G)$ is canonically $\cO_X^G$-linear, and thus so is its inverse. By the fact that $F_{\ell,\ell',\delta}$ restricts to the identity functor over $X^{\rm s}(0)/G$, which is contained in both $X^{\rm{ss}}(\ell)/G$ and $X^{\rm{ss}}(\ell')/G$, we mean that $\op{res}_{X^{\rm s}(0)} \circ F_{\ell,\ell',\delta} \simeq \op{res}_{X^{\rm s}(0)}$. This is evident from the canonical equivalences
\[
\op{res}_{X^{\rm s}(0)} \circ F_{\ell,\ell',\delta} \simeq \op{res}_{X^{\rm s}(0)} \circ \op{res}_{X^{\rm{ss}}(\ell)}^{-1} \simeq \op{res}_{X^{\rm s}(0)} \circ \op{res}_{X^{\rm{ss}}(\ell)} \circ \op{res}_{X^{\rm{ss}}(\ell)}^{-1}. \qedhere
\]
\end{proof}

In fact, we can say more about the structure of these equivalences. For any $\ell,\delta \in M_\bR^W$, we consider the following locally free sheaf
\[
\cU_{\ell,\delta} := \bigoplus_{\substack{U \in \Irr(G)\\ \op{Char}(U) \subseteq \delta + \ol{\nabla}}} \cO_X \otimes U|_{X^{\rm{ss}}(\ell)} \in {\rm D}^b(X^{\rm{ss}}(\ell)/G).
\]

\begin{cor}\label{cor:tilting}
If $(M_\bR^W)_{\ol{\Sigma}\text{-{\rm gen}}} \neq \emptyset$, then for any $\ell, \delta \in M_\bR^W$ with $X^{\rm{ss}}(\ell)=X^{\rm{s}}(\ell)$ and $\partial(\delta + \ol{\nabla}) \cap M = \emptyset$, the vector bundle $\cU_{\ell,\delta}$ is a tilting generator for ${\rm D}^b(X^{\rm{ss}}(\ell)/G)$. The equivalence $F_{\ell, \ell', \delta}$ maps $\cU_{\ell,\delta}$ to $\cU_{\ell',\delta}$.
\end{cor}
\begin{proof}
The fact that $\cU_{\ell,\delta}$ is a generator follows from the essential surjectivity in \autoref{thm:magic_windows}. The fact that it is a tilting bundle (meaning it has no higher self-extensions) follows from the fully faithfulness in \autoref{thm:magic_windows} and the fact that $X$ is affine and $G$ is reductive so there are no higher self-extensions in ${\rm D}^b(X/G)$.
\end{proof}

\begin{rem}
The fact that $\cU_{\ell,\delta}$ is a tilting generator for ${\rm D}^b(X^{\rm{ss}}(\ell)/G)$ implies that this category is equivalent to $\Lambda\text{-Mod}$, where $\Lambda = \op{Hom}_{X/G}(\cU_{\ell,\delta}, \cU_{\ell,\delta})$ is precisely the non-commutative $\cO_X^G$-algebra which was shown to be a non-commutative resolution of $\Spec(\cO_X^G)$ in \cite{SVdB}. Thus in the quasi-symmetric case, we have shown that their non-commutative resolutions are in fact \emph{commutative}.\end{rem}

\begin{rem} [Fourier--Mukai kernels]
Let $\delta$,$\ell$, and $\ell'$ be as in \autoref{cor:tilting} and \autoref{cor:equivalences}. \autoref{thm:magic_windows} implies that $\Lambda = \op{End}_{X/G}(\cU_{\ell,\delta}) = \op{End}_{X/G}(\cU_{\ell',\delta})$. We regard $\cU_{\ell',\delta}$ as an equivariant left $\cO_{X^{\rm{ss}}(\ell')} \otimes_k \Lambda$-module, and $\RHom_{X^{\rm{ss}}(\ell)/G}(\cU_{\ell,\delta},E)$ as a complex of right $\Lambda$-modules for any $E \in {\rm D}^b(X^{\rm{ss}}(\ell)/G)$. \autoref{cor:tilting} allows one to express the derived equivalence $F_{\ell,\ell',\delta}$ of \autoref{cor:equivalences} via the explicit formula
\[
F_{\ell,\ell',\delta}(E) = \cU_{\ell',\delta} \otimes_\Lambda \RHom_{X^{\rm{ss}}(\ell)/G}(\cU_{\ell,\delta},E).
\]
Regarding the dual vector bundle $\cU_{\ell,\delta}^\ast$ as an equivariant right $\cO_{X^{\rm{ss}}(\ell)} \otimes_k \Lambda$-module, the Fourier--Mukai kernel for $F_{\ell,\ell',\delta}$ is the derived tensor product $\cU_{\ell,\delta}^\ast \otimes_\Lambda \cU_{\ell',\delta} \in {\rm D}^b(X^{\rm{ss}}(\ell) / G \times X^{\rm{ss}}(\ell') / G)$.

For an alternative description of the Fourier--Mukai kernel: Let $\Delta \colon X^{\rm{ss}}(\ell) / G \to X^{\rm{ss}}(\ell) / G \times X^{\rm{ss}}(\ell)/G$ be the diagonal morphism. It is finite if $X^{\rm{ss}}(\ell) = X^{\rm{s}}(\ell)$. Therefore $\cF := \Delta_\ast(\cO_{X^{\rm{ss}}(\ell)})$ is a coherent sheaf, and applying \autoref{thm:magic_windows} to $X\times X / G \times G$ provides an algorithm to extend this uniquely and functorially to a complex $\tilde{\cF} \in \cM((\delta,\delta) + \bar{\nabla} \times \bar{\nabla}) \subset {\rm D}^b(X \times X / G\times G)$. The restriction of $\tilde{\cF}$ to $X^{\rm{ss}}(\ell) \times X^{\rm{ss}}(\ell')$ is the Fourier--Mukai kernel for $F_{\ell,\ell',\delta}$ (see \cite[\S 2.3]{HL}).
\end{rem}

\subsubsection{The non-local version} \label{sect:non_local}


Via the Luna slice theorem, our main theorem generalizes to a ``global'' version. This will be discussed in detail in forthcoming work of the first author with D. Maulik and A. Okounkov. Consider a smooth $G$-variety $X$ such that $X/G$ admits a good quotient $X/\!/G$ and which is locally quasi-symmetric in the sense that for any $x \in X$ with a closed orbit, the tangent space $T_x X$ is a quasi-symmetric representation of the stabilizer group $\op{Stab}_G(x)$. Then for any line bundle $\ell \in \op{Pic}_G(X)_\bR$ one can define a GIT-semistable locus $X^{\rm{ss}}(\ell)$ using the Hilbert--Mumford criterion, and $X^{\rm{ss}}(\ell)/G$ admits a good quotient $X^{\rm{ss}}(\ell)/\!/G$ which is projective over $X/\!/G$. Then the global version of \autoref{cor:equivalences} provides derived equivalences ${\rm D}^b(X^{\rm{ss}}(\ell)/G) \to {\rm D}^b(X^{\rm{ss}}(\ell')/G)$ for any two $\ell,\ell' \in \op{Pic}_G(X)_\bR$ satisfying a suitable genericity hypothesis. This set up includes many known examples of derived equivalences arising from variation of GIT quotient, including the ``balanced wall crossings'' studied in \cite{HL}, and in particular all derived equivalences arising from a generic variation of GIT quotient for the action of a torus on a smooth projective-over-affine variety.

\subsection{Bases in $K$-theory}

Next we observe that \autoref{thm:magic_windows} can be used to describe a basis for the $K$-theory of the quotient stack $X^{\rm{ss}}/G$. Because the action of $\bG_m$ on $X$ by scaling commutes with the action of $G$, the open $G$-semistable locus $X^{\rm{ss}}(\ell) \subset X$ is $G \times \bG_m$ equivariant. For any locally free sheaf $E$ on a stack $\cX$, we let $[E]$ denote the corresponding class in $K_0(\Perf(\cX))$. Note that $K_0(\Perf(X^{\rm{ss}}(\ell)/G\times \bG_m))$ is a module over $K_0(\Rep(\bG_m)) \cong \bZ[t^\pm]$ via tensor product with the tautological character of $\bG_m$.

\begin{prop} \label{prop:K_theory}
With notation as above, if $(M_\bR^W)_{\ol{\Sigma}\text{-{\rm gen}}} \neq \emptyset$ and $\ell,\delta \in \op{Pic}(BG)_\bR$ with $X^{\rm{ss}}(\ell) = X^{\rm{s}}(\ell)$ and $\partial(\delta + \ol{\nabla}) = \emptyset$, then $K_i(\Perf(X^{\rm{ss}}(\ell)/G))$ {\rm (}respectively, $K_i(\Perf(X^{\rm{ss}}(\ell) / G\times \bG_m))${\rm )} vanishes for $i \neq 0$, and for $i=0$ is a free $\bZ$-module {\rm (}respectively, is a free $\bZ[t^\pm]$-module{\rm )} with basis
\[
\{ [\cO_X \otimes U] \mid U \in \Irr(G) \text{ and } \op{Char}(U) \subseteq \delta + \ol{\nabla} \}.
\]
When the base field is $k = \bC$, then the same statements hold for topological $K$-theory {\rm (}of the analytification{\rm )}, and the orbifold Chern characters $\op{ch}([\cO_X \otimes U])$ provide a basis for the rational cohomology of the inertia stack of $X^{\rm{ss}}(\ell)/G$.
\end{prop}

\begin{proof}
The non-equivariant statement follows from the $\bG_m$-equivariant one and the fact that
\[
K_\bullet({\rm D}^b(X^{\rm{ss}}(\ell)/G)) \cong K_\bullet ({\rm D}^b(X^{\rm{ss}}(\ell) / G\times \bG_m)) \stackrel{\rm L}{\otimes}_{\bZ[t^\pm]} \bZ[t^\pm] / (t-1).
\]
So we will just prove the $\bG_m$-equivariant statement.

Let $U\langle n \rangle$ denote the representation of $G \times \bG_m$ whose restriction to $G$ is $U$ and on which $\bG_m$ acts with weight $-n$. The vector bundles $\cO_X \otimes U\langle n\rangle$ form a full exceptional collection for the category ${\rm D}^b(X / G\times \bG_m)$ since
\[
\RHom_{X/G\times\bG_m}(\cO_X \otimes U\langle m\rangle, \cO_X \otimes V\langle n\rangle) = \begin{cases} 0 & \text{if } n > m \\  
(\op{Sym}^{m-n}(X^\dual) \otimes U^\dual \otimes V)^G& \text{else} \end{cases}.
\]
The fully-faithfulness of \autoref{lem:fully_faithful} implies that the restriction functor ${\rm D}^b(X/G\times \bG_m) \to {\rm D}^b(X^{\rm{ss}}(\ell)/G\times \bG_m)$ is fully faithful when restricted to the subcategory generated by $\cO_X \otimes U\langle n\rangle$ with $\op{Char}(U) \subseteq \delta + \ol{\nabla}$. Therefore the objects $\cO_X \otimes U\langle n \rangle|_{X^{\rm{ss}}(\ell)}$ with $\op{Char}(U) \subseteq \delta + \ol{\nabla}$ form an exceptional collection in ${\rm D}^b(X^{\rm{ss}}(\ell)/G\times \bG_m)$. To show that this exceptional collection is full, it suffices to show that these objects split-generate ${\rm D}^b(X^{\rm{ss}}(\ell)/G\times \bG_m)$. The map $p \colon X^{\rm{ss}}(\ell)/G\times \bG_m \to X^{\rm{ss}}(\ell)/G$ is faithfully flat and affine, so the pushforward of a generating set is a generating set, and $p_\ast (\cO_X \otimes U) = \bigoplus_{n \in \bZ} \cO_X \otimes U\langle n \rangle$.

The existence of a full exceptional collection implies the vanishing of $K_i({\rm D}^b(X^{\rm{ss}}(\ell)/G \times \bG_m))$ for $i \neq 0$ and that the classes of the objects $[\cO_X\otimes U\langle n \rangle]$ form a basis for $K_0({\rm D}^b(X^{\rm{ss}}(\ell)/G\times \bG_m))$ as a free $\bZ$-module. Because tensoring with the tautological character of $\bG_m$ maps $\cO_X \otimes U\langle n \rangle \mapsto \cO_X \otimes U \langle n+1 \rangle$, it follows that the classes $[\cO_X \otimes U\langle 0 \rangle]$ form a basis for $K_0({\rm D}^b(X^{\rm{ss}}(\ell)/G\times \bG_m))$ as a free $\bZ[t^\pm]$-module.

\medskip
\noindent \textit{When the base field $k=\bC$}:
For both $X^{\rm{ss}}(\ell)/G$ as well as for $X^{\rm{ss}}(\ell)/G\times \bG_m$, one can recover the equivariant topological $K$-theory from the derived category via Blanc's topological $K$-theory of dg-categories \cite{HLPomerleano, Blanc}. As above, the full exceptional collection in ${\rm D}^b(X^{\rm{ss}}(\ell)/G\times\bG_m)$ implies that $K^{\rm top}({\rm D}^b(X^{\rm{ss}}(\ell)/G\times \bG_m))$ has no homology outside of degree $0$ and that the classes $[\cO_X \otimes U\langle 0 \rangle]$ form a basis for $K^{\rm top}({\rm D}^b(X^{\rm{ss}}(\ell)/G\times \bG_m)) \cong K^0_{U(1)}(X^{\rm{ss}}(\ell)/G)$ as a module over $\bZ[t^\pm]$. The claim for the non $\bG_m$-equivariant $K$-theory follows from the fact that for topological $K$-theory we have 
\[
K(X^{\rm{ss}}(\ell)/G) \cong K_{U(1)}(X^{\rm{ss}}(\ell)/G) \stackrel{\rm L}{\otimes}_{\bZ[t^\pm]}\bZ[t^\pm]/(t-1).
\]
Finally, the topological $K$-theory of a smooth and proper Deligne--Mumford stack with rational coefficients is identified with the rational cohomology of the inertia stack via the (orbifold) Chern character \cite{HLPomerleano}.
\end{proof}

\begin{rem}
In many examples where $X^{\rm{ss}}(\ell)/G$ is a scheme, it is known that this scheme has a stratification by affine spaces and thus its algebraic and topological $K$-theory are free with a generating set corresponding to the strata. \autoref{prop:K_theory} takes a very different approach, where the basis is determined by the representation theory of $G$, and it applies to a broader class of examples, including Deligne--Mumford GIT quotients.
\end{rem}

\section{Symplectic resolutions by hyperk\"ahler quotients}
\label{sect:hyperkaehler}

In this section we consider a symplectic linear representation $X$ of a reductive group $G$, which comes equipped with an algebraic $G$-equivariant moment map $\mu \colon X \to \fg^\dual$. By definition $d\mu(x)$ is a $1$-form valued in $\fg^\dual$ such that
\begin{equation}\label{eqn:moment_map}
\langle d\mu(x),\xi \rangle = \omega(-,H_\xi) \in \Gamma(X,\Omega^1_X),
\end{equation}
where $\omega$ is the symplectic form on $X$ and $H_\xi$ is the vector field generated by $\xi \in \fg$, and this defines $\mu$ up to a shift by a constant in $(\fg^\dual)^G$. We then specify $\mu$ uniquely by requiring that $\mu$ is equivariant with respect to the scaling action on $X$ and the scaling action of weight $2$ on $\fg^\dual$. We denote $X_0 := \mu^{-1}(0)$, and we consider the hyperk\"ahler quotient, which we define to be $X_0^{\rm{ss}}(\ell) / G$ for some character $\ell \in \op{Pic}(BG)_\bR \cong M_\bR^W$.

If $X^{\rm{ss}}(\ell) / G$ is Deligne--Mumford, then the moment map restricted to $X^{\rm{ss}}(\ell)$ is smooth, because \eqref{eqn:moment_map} implies that the critical points of $\mu$ are precisely those points which have positive dimensional stabilizers. Therefore $X_0^{\rm{ss}}(\ell) / G \to \Spec(\cO_{X_0}^G)$ is a smooth symplectic resolution of singularities (possibly by an orbifold). The following theorem shows that all such resolutions of $\Spec(\cO_{X_0}^G)$ resulting from a generic choice of $\ell \in M_\bR^W$ will have equivalent derived categories in this case.

\begin{thm} \label{thm:hyperkaehler}
Let $X$ be a symplectic representation of a reductive group $G$, and let $\ol{\Sigma}$ be the zonotope corresponding to the character of the linear representation $X \oplus \fg$. If $(M_\bR^W)_{\ol{\Sigma}\text{-{\rm gen}}} \neq \emptyset$, then for any $\ell, \ell' \in \op{Pic}(BG)_\bR$ such that $X_0^{\rm{ss}}(\ell) = X_0^{\rm{s}}(\ell)$ and $X_0^{\rm{ss}}(\ell') = X_0^{\rm{s}}(\ell')$ one has a derived equivalence of hyperk\"ahler quotients ${\rm D}^b(X_0^{\rm{ss}}(\ell) / G) \simeq {\rm D}^b(X_0^{\rm{ss}}(\ell')/G)$.
\end{thm}

We will prove this theorem at the end of this section. It is a consequence of a more general result about (graded) categories of singularities which follows formally from \autoref{thm:magic_windows}. We consider a linear representation $X$ of $G \times \bG_m$ which is quasi-symmetric for the action of $G$, and we let $W \colon X \to \bA^1$ be a function which is $G$-invariant and $\bG_m$-equivariant for the scaling action of $\bG_m$ on $\bA^1$. With this setup we briefly recall the construction and basic properties of the graded category of singularities, originally introduced in \cite{orlov}. Many authors have studied versions of the category of graded singularities \cite{Isik,BFK,Shipman,hirano}, and these methods are by now standard in the theory of derived categories, but we use the conventions of \cite{HLPomerleano}, which contains formulations of these standard results in the ways we will need.

There is a natural transformation $- \otimes \cL [-2] \to \op{id}$ where $\cL = \cO_{W^{-1}(0)} \langle -1 \rangle$ (both are endofunctors of ${\rm D}^b(W^{-1}(0)/G\times \bG_m)$). We define the graded category of singularities ${\rm D}^b_{\op{sing}}(X/G\times \bG_m,W)$ to be the idempotent completion of the dg-category whose objects are objects of ${\rm D}^b(W^{-1}(0)/ G\times \bG_m)$ and whose morphisms are
\[
\RHom_{{\rm D}^b_{\op{sing}}(X/G\times \bG_m,W)}(E,F) :=  \op{hocolim}_p \left( \RHom_{{\rm D}^b(W^{-1}(0)/G\times \bG_m)}(E,F \otimes \cL^{-p})[2p] \right).
\]
This is referred to as the graded category of singularities because when one collapses the $\bZ$-grading on ${\rm D}^b_{\op{sing}}(X/G \times \bG_m,W)$ to a $\bZ/2\bZ$-grading, one gets the usual derived category of singularities on $W^{-1}(0)/G$. More precisely, if $k(\!(\beta)\!)$ is the field of Laurent series on a variable of homological degree $-2$, then by \cite[Proposition 1.22]{HLPomerleano} we have\footnote{The map $X/G\times \bG_m \to \bA^1/\bG_m$ is a special case of what is sometimes referred to as a graded Landau--Ginzburg model: a smooth quasi-compact stack $\cX$ whose points have affine automorphism groups along with a map $\cX \to \bA^1 / \bG_m$. The definition of ${\rm D}^b_{\op{sing}}(\cX,W)$ carries over to this setting, and we have ${\rm D}^b_{\op{sing}}(\cX,W) \otimes_k k(\!(\beta)\!) \simeq {\rm D}^b((W')^{-1}(0)) / \Perf((W')^{-1}(0))$, where $W'$ is the pullback of $W$ to the total space of the $\bG_m$-torsor $\cX' \to \cX$ classified by the map $\cX \to \bA^1/\bG_m \to B\bG_m$.}
\[
{\rm D}^b_{\op{sing}}(X/G\times \bG_m,W) \otimes_k k(\!(\beta)\!) \simeq {\rm D}^b(W^{-1}(0)/G) / \Perf(W^{-1}(0)/G),
\]
where the latter denotes the Verdier-Keller-Drinfeld quotient of pre-triangulated dg-categories (regarded as a stable $k$-linear $\infty$-category).

\begin{cor} \label{cor:sing}
Let $W \colon X / G\times \bG_m \to \bA^1 / \bG_m$ be as above. Let $\ol{\Sigma} \subset M_\bR$ denote the zonotope associated to the character of the quasi-symmetric representation $X$ of $G$, and assume that $\ol{\Sigma}$ spans $M_\bR$ and that $(M_\bR^W)_{\ol{\Sigma}\text{-{\rm gen}}} \neq \emptyset$. Then for any $\ell,\ell'$ such that $X^{\rm{ss}} = X^{\rm{s}}$ we have an equivalence
\[
{\rm D}^b_{\op{sing}}(X^{\rm{ss}}(\ell)/ G\times \bG_m ,W) \simeq {\rm D}^b_{\op{sing}}(X^{\rm{ss}}(\ell')/G\times \bG_m,W).
\]
\end{cor}

\begin{proof}
The argument in the proof of \autoref{prop:K_theory} shows that \autoref{thm:magic_windows}, and thus \autoref{cor:equivalences}, extend to the setting in which there is an auxiliary group action (in this case, a $\bG_m$-action) commuting with the action of $G$ on $X$. The only modification is that we must redefine $\cM(\delta + \ol{\nabla})$ to be the subcategory of ${\rm D}^b(X/G\times \bG_m)$ generated by $\cO_X \otimes U\langle n \rangle$ for $U \in \Irr(G)$ with $\op{Char}(U) \subseteq \delta + \ol{\nabla}$. Thus for any $\delta \in M_\bR^W$ with $\partial(\delta + \ol{\nabla}) \cap M = \emptyset$ we have an equivalence 
\begin{equation}\label{eqn:equivalence}
F_{\ell,\ell',\delta} \colon {\rm D}^b(X^{\rm{ss}}(\ell)/G\times \bG_m) \xrightarrow{\op{res}_{X^{\rm{ss}}(\ell)}^{-1}} \cM(\delta + \ol{\nabla}) \xrightarrow{\op{res}_{X^{\rm{ss}}(\ell)}} {\rm D}^b(X^{\rm{ss}}(\ell')/G\times \bG_m).
\end{equation}

The argument for ${\rm D}^b_{\op{sing}}$ follows formally from this and is essentially the same as in \cite[Proposition 5.5]{HL}, so we only sketch it here briefly: The key fact is that all three categories in \eqref{eqn:equivalence} are module categories over the symmetric monoidal $\infty$-category $\Perf(\bA^1 / \bG_m)^\otimes$ via the $G$-invariant and $\bG_m$-equivariant map $W \colon X \to \bA^1$, and the restriction functors are maps of $\Perf(\bA^1/\bG_m)^{\otimes}$-module categories.\footnote{The fact that $\cM(\delta + \ol{\nabla}) \subset {\rm D}^b(X / G \times \bG_m)$ is closed under tensor product with $W^\ast E$ for $E \in \Perf(\bA^1 / \bG_m)$ implies that it is a $\Perf(\bA^1/\bG_m)^\otimes$-module subcategory. This can be checked in two ways: 1) one can use the explicit generators used to define $\cM(\delta + \ol{\nabla})$ and the fact that $\Perf(\bA^1/\bG_m)$ is generated by the locally free sheaves $\cO_{\bA^1}\langle n \rangle$, or 2) one can use the fact that $\cM(\delta + \ol{\nabla})$ is equal to $\cG^w$ for some $w$, and the grade restriction rules used to define $\cG^w$ are preserved under tensor product with objects in $W^\ast(\Perf(\bA^1 / \bG_m))$.} Note that $\bA^1 / \bG_m$ is a perfect stack, and the inclusion of the origin $\{0\} / \bG_m \hookrightarrow \bA^1 / \bG_m$ is a proper representable map. The hypotheses of \cite[Theorem 3.0.4]{BZPN} are satisfied, so the ``$\ast$-integral transform'' construction provides an equivalence\footnote{Technically the left hand side must be interpreted as the derived category, with bounded coherent homology, of the \emph{derived} zero fiber, but because $X$ is integral this will agree with the classical zero fiber as long as $W \colon X \to \bA^1$ is surjective.}
\[
{\rm D}^b (W^{-1}(0) /G \times \bG_m) \simeq \op{Fun}^{\rm ex}_{\Perf(\bA^1/\bG_m)^\otimes} (\Perf(\{0\}/\bG_m), {\rm D}^b(X / G\times \bG_m)),
\]
where the right hand side is the $\infty$-category of exact functors of $\Perf(\bA^1/\bG_m)^\otimes$-module categories. The analogous equivalence holds for ${\rm D}^b(W^{-1}(0)^{\rm{ss}}(\ell)/G)$ and ${\rm D}^b (W^{-1}(0)^{\rm{ss}}(\ell')/G)$. 

Therefore, applying $\op{Fun}^{\rm ex}_{\Perf(\bA^1/\bG_m)^\otimes}(\Perf(\{0\}/\bG_m), -)$ to the equivalence \eqref{eqn:equivalence} induces an equivalence
\[
F_{\ell,\ell',\delta} \colon {\rm D}^b(W^{-1}(0)^{\rm{ss}}(\ell)/G) \simeq {\rm D}^b (W^{-1}(0)^{\rm{ss}}(\ell')/G).
\]
One can check that this equivalence canonically preserves the natural transformation $- \otimes \cO_{W^{-1}(0)}\langle-1\rangle[-2] \to \op{id}$ used to define the graded category of singularities, as both are the restrictions to $W^{-1}(0)^{\rm{ss}}(\ell) / G\times \bG_m$ and $W^{-1}(0)^{\rm{ss}}(\ell') / G\times \bG_m$ of the same structure on $W^{-1}(0) / G \times \bG_m$.\footnote{In fact, this natural transformation can be encoded entirely in terms of the $\Perf(\bA^1/\bG_m)^\otimes$-module structure of ${\rm D}^b(X/G\times\bG_m)$ and the description of ${\rm D}^b(W^{-1}(0) / G\times \bG_m)$ as a functor category.} It follows that $F_{\ell,\ell',\delta}$ induces an equivalence for ${\rm D}^b_{\op{sing}}$ as well.
\end{proof}

Now consider a smooth Deligne--Mumford global quotient stack $\cX$, and let $\sigma \in \Gamma(\cX,\cE)$ be a section of a locally free sheaf. Then one can consider the function $W \colon \op{Tot}(\cE^\dual) \to \bA^1$ induced by $\sigma$, which is equivariant with respect to the scaling action on $\op{Tot}(\cE^\dual)$. Then we have
\begin{prop} \label{lem:LG_CY}
There is a canonical equivalence ${\rm D}^b_{\op{sing}}(\op{Tot}(\cE^\dual) / \bG_m, W) \simeq {\rm D}^b(\sigma^{-1}(0))$.
\end{prop}
\begin{proof}
This is a slight generalization of the main results of \cite{Isik, Shipman, orlov2}, which establish the case where $G= \{1\}$. In the equivariant setting, it is established in \cite[Proposition 4.8]{hirano2} for factorization categories, which is equivalent to the claim for singularity categories by \cite[Theorem 3.6]{hirano2} (see also the proof of \cite[Theorem 4.2]{hirano2}).
\end{proof}

\begin{lem} \label{lem:crit_locus}
Assume that $X^{\rm{ss}}(\ell) / G$ is Deligne--Mumford. Then we have
\[
\Crit(W) \cap (X \times \fg)^{\rm{ss}}(\ell) = \Crit(W) \cap (X^{\rm{ss}}(\ell) \times \fg).
\]
\end{lem}
\begin{proof}
The defining property of $\mu$ in \eqref{eqn:moment_map} implies that
\[
dW(x,\xi) = \omega_x(-,(H_\xi)_x) + \mu(x),
\]
where the first and second summand correspond to the first and second summand in the decomposition $T^\ast_{(x,\xi)} (X \times \fg) \cong T^\ast_x X \oplus \fg^\dual$. Therefore $dW = 0$ if and only if $\mu(x) = 0$ and $(H_\xi)_x = 0$, i.e., the infinitesimal action of $\xi$ fixes the point $x$.

If $x$ is an unstable point of $X$, however, then we know from GIT that $\op{Stab}(x) \subseteq P_\lambda$, where $\lambda$ is the maximally destabilizing one-parameter subgroup of $x$. This implies that $\xi \in \op{Lie}(P_\lambda)$, so in fact $\lambda(t) \cdot (x,\xi)$ has a limit as $t \to 0$ and thus $\lambda$ destabilizes the point $(x,\xi)$ as well.
\end{proof}

\begin{lem} \label{lem:restriction}
Let $\cX$ be a smooth perfect stack, and let $W \colon \cX \to \bA^1 / \bG_m$ be a morphism. Assume that $\cZ \subset \cX$ is a closed substack such that $\cZ \cap \Crit(W) = \emptyset$. Then the restriction functor ${\rm D}^b_{\op{sing}}(\cX,W) \to {\rm D}^b_{\op{sing}}(\cX \setminus \cZ,W)$ is an equivalence of categories.
\end{lem}

Scheme theoretic versions of this lemma go back to the first study of categories of matrix factorizations in the geometric context \cite[Proposition 1.14]{orlov}, but we include a proof in this more general context for completeness.

\begin{proof}
As remarked above, the $2$-periodization of ${\rm D}^b_{\op{sing}}(\cX,W)$ is canonically identified with the Drinfeld--Verdier quotient ${\rm D}^b((W')^{-1}(0)) / \Perf((W')^{-1}(0))$ where $W'$ is the restriction of $W$ to the total space of the $\bG_m$-torsor $\cX' \to \cX$ classified by the map $\cX \to \bA^1/\bG_m \to B\bG_m$. One can use this description to prove the corresponding claim for the restriction functor on the $2$-periodization of these categories (see, for instance, \cite[Proposition 4.1.6]{Preygel}). Since the fully-faithfulness of the restriction functor can be verified after $2$-periodization, it follows that the restriction functor is fully faithful. Finally, a fully faithful functor between idempotent complete dg-categories is essentially surjective if and only if its essential image contains a split generating set. By definition, ${\rm D}^b_{\op{sing}}(\cX \setminus \cZ,W)$ is split-generated by objects corresponding to objects of ${\rm D}^b(\cX \setminus \cZ)$, and any such object can be extended to an object of ${\rm D}^b(\cX)$.
\end{proof}

\begin{proof}[Proof of \autoref{thm:hyperkaehler}]
First note that by \autoref{cor:perturbation} we may make an arbitrarily small perturbation of $\ell$ and $\ell'$, without changing $X_0^{\rm{ss}}(\ell)$ or $X_0^{\rm{ss}}(\ell')$, such that both $(X \times \fg)^{\rm{ss}}(\ell) = (X \times \fg)^{\rm{s}}(\ell)$ and $(X \times \fg)^{\rm{ss}}(\ell') = (X \times \fg)^{\rm{s}}(\ell')$. It therefore suffices to assume this throughout the remainder of the proof.

We apply the preceding observations to the $G$-invariant $\bG_m$-equivariant map $W \colon X \times \fg \to \bA^1$ induced by the moment map $\mu \colon X \to \fg^\dual$. \autoref{cor:sing} implies that we have an equivalence 
\[
F_{\ell,\ell',\delta} \colon {\rm D}^b_{\op{sing}}((X\times \fg)^{\rm{ss}}(\ell) / G \times \bG_m,W) \simeq {\rm D}^b_{\op{sing}}((X \times \fg)^{\rm{ss}}(\ell') / G \times \bG_m,W).
\]
Combining \autoref{lem:restriction} with \autoref{lem:crit_locus} shows that the restriction functor induces an equivalence
\[
{\rm D}^b_{\op{sing}}((X\times \fg)^{\rm{ss}}(\ell) / G \times \bG_m, W) \simeq {\rm D}^b_{\op{sing}}((X^{\rm{ss}}(\ell) \times \fg) / (G \times \bG_m) , W).
\]
Finally when we regard $X^{\rm{ss}}(\ell) \times \fg / G \to X^{\rm{ss}}(\ell)/G$ as a vector bundle, the function $W$ on the right hand side above is induced under the construction preceding \autoref{lem:LG_CY} by the function $\mu$ regarded as a section of $\cO_X \otimes \fg^\dual$, so we have an equivalence
\[
{\rm D}^b_{\op{sing}}((X^{\rm{ss}}(\ell) \times \fg) / (G \times \bG_m) , W) \simeq {\rm D}^b((\mu^{-1}(0))^{\rm{ss}}(\ell) / G).
\]
The same argument gives an equivalence 
\[
{\rm D}^b_{\op{sing}}((X \times \fg)^{\rm{ss}}(\ell') / (G \times \bG_m) , W) \simeq {\rm D}^b((\mu^{-1}(0))^{\rm{ss}}(\ell') / G),
\]
which, when combined with $F_{\ell,\ell',\delta}$, provides the desired equivalence.
\end{proof}

\subsection{Example: Nakajima quiver varieties}

A large source of applications of Theorem~\ref{thm:hyperkaehler} comes from Nakajima quiver varieties, and we now explain what it gives in this setting. We partially follow the exposition in \cite[\S 2]{nakajima}.

Our initial data is a finite graph $Q$. Let $I$ denote the set of vertices, and let $E$ denote the set of edges. Let $H$ denote the set of pairs consisting of an edge and a choice of orientation, so $\#H = 2\# E$. For $h \in H$, let $o(h)$ denote the outgoing vertex and $i(h)$ denote the ingoing vertex (so $h$ points away from $o(h)$ and towards $i(h)$). 

Given $I$-graded vector spaces $V^1$ and $V^2$, define
\[
L(V^1, V^2) = \bigoplus_{i \in I} \op{Hom}(V^1_i, V^2_i), \qquad E(V^1, V^2) = \bigoplus_{h \in H} \op{Hom}(V^1_{o(h)}, V^2_{i(h)}).
\]
Let $\bv = (v_i)_{i \in I}$ and $\bw = (w_i)_{i \in I}$ be two sequences of non-negative integers and let $V$ and $W$ be $I$-graded vector spaces of dimensions $\bv$ and $\bw$, respectively. Define
\[
\bM_Q(\bv, \bw) = E(V,V) \oplus L(W,V) \oplus L(V,W), \qquad G_Q(\bv) = \prod_{i \in I} {\bf GL}(V_i).
\]
Denote elements by $(B,a,b)$. We omit the subscript if not needed.
For $h \in H$, let $\ol{h}$ be the same edge with reversed orientation. Choose an orientation $\Omega \subset H$, which means that $\Omega \cap \ol{\Omega} = \emptyset$ and $\Omega \cup \ol{\Omega} = H$. Then $\Omega$ defines a function $\varepsilon \colon H \to \{\pm 1\}$ by $\varepsilon(h) = 1$ if $h \in \Omega$ and $-1$ otherwise. Finally, this gives a symplectic form on $\bM(\bv, \bw)$ by
\[
\omega((B,a,b), (B',a',b')) = \sum_{h \in H} \op{trace}(\varepsilon(h) B_h B'_{\ol{h}}) + \sum_{i \in I} \op{trace}(a_ib'_i - a'_ib_i).
\]
The group $G$ acts on $\bM(\bv, \bw)$ in the obvious way and preserves the form $\omega$. This gives an algebraic moment map $\mu \colon \bM(\bv, \bw) \to \fg(\bv)$ where $\fg(\bv)$ is the Lie algebra of $G(\bv)$ (here we are identifying $\fg(\bv)$ with its dual via the trace form). In the $\mathfrak{gl}(v_i)$-component, $\mu(B,a,b)$ takes the value 
\[
\mu(B,a,b)_i = \sum_{\substack{h \in H\\ i(h) = i}} (\varepsilon(h) B_h B_{\ol{h}} + a_i b_i).
\]

A stability parameter $\zeta$ is given by a real sequence $(\zeta_i)_{i \in I}$. We have $\op{Pic}(B G_Q(\bv)) \cong \bR^I$, and the stability parameter $\zeta$ corresponds to the parameter we have denoted $\ell$ above, but we use $\zeta$ here for consistency with \cite{nakajima}. Nakajima provides an interpretation in \cite[\S 2.2]{nakajima} of $\zeta$-semistability and $\zeta$-stability in terms of representations of $Q$, but we will not need it. 

\begin{prop}
With the notation above, let $\ol{\Sigma}$ denote the zonotope determined by the character of the representation $\bM_Q(\bv,\bw) \times \fg_Q(\bv)$ of $G = G_Q(\bv)$. Consider the graph $\Gamma$ obtained from $Q$ by deleting all vertices $i$ of $Q_0$ such that $v_i = 0$, and assume that each connected component of $\Gamma$ has a vertex $i$ such that $w_i \ne 0$. Then $(M^W_{\bR})_{\ol{\Sigma}\text{{\rm-gen}}} \ne \emptyset$ and $\ol{\Sigma}$ spans $M_\bR$.
\end{prop}

\begin{proof}
Let $\zeta$ be a stability parameter. We first claim that our conditions on $\bv$ and $\bw$ show that a generic point in $\bM_Q(\bv,\bw) \times \fg_Q(\bv)$ has trivial stabilizer under $T$. First, let $H$ be a maximal torus in ${\bf GL}(V)$ and consider the adjoint action on the Lie algebra $\mathfrak{gl}(V)$. Then the stabilizer in $H$ of a generic element in $\mathfrak{gl}(V)$ is the center of ${\bf GL}(V)$. So choose a point in $\fg_Q(\bv)$ with this genericity assumption with respect to our fixed maximal torus $T \subset G_Q(\bv)$. The stabilizer must belong to the center of $G_Q(\bv)$. But note that if $w_i \ne 0$, then any element of the stabilizer that also lives in the central $\bG_m$ acting on vertex $i$ must act trivially. This implies that the same is true for any vertex in the same connected component (in $\Gamma$) as such a vertex, and so our assumption guarantees that a generic point has trivial stabilizer.

Now we claim that $\zeta \in \op{Pic}(BG)_\bR$ is generic for $\ol{\Sigma}$ if $\theta \cdot \zeta \ne 0$ for all $(\theta_i)_{i \in I} \ne 0$ such that $\theta_i \le v_i$ for all $i \in I$. First, since $\bM(\bv, \bw) \times \fg(\bv)$ is a $T(\bv)$-equivariant subvariety of  $\bM(\bv,\bw) \times \fg(\bv) \times \fg(\bv)^\dual$, it suffices to find conditions for $(\bM(\bv, \bw) \times \fg(\bv) \times \fg(\bv)^\dual)^{\zeta\text{\rm{-ss}}}/T(\bv)$ to be Deligne--Mumford since the Hilbert--Mumford criterion implies that a point in a closed $G$-equivariant subspace is semistable if and only if it is semistable in the ambient space.

Define a new graph $Q'$ as follows. First, add an extra loop at every vertex of $Q$. Second, replace each vertex $i$ with $v_i$ vertices $i(1), \dots, i(v_i)$. For each edge between $i$ and $j$ (including the case $i=j$), add an edge between $i(\alpha)$ and $j(\beta)$ for all $1 \le \alpha \le v_i$ and $1 \le \beta \le v_j$. Now set $v'_{i(\alpha)} = 1$ and $w'_{i(\alpha)} = w_i$ for all $i \in I$ and $1 \le \alpha \le v_i$. Let $I'$ denote the vertex set of $Q'$. Then we have
\[
\bM_Q(\bv, \bw) \times \fg_Q(\bv) \times \fg_Q(\bv)^\dual = \bM_{Q'}(\bv', \bw'), \qquad G_{Q'}(\bv') = T_Q(\bv).
\]
Given a stability parameter $\zeta$ for $Q$, define a stability parameter $\zeta' = (\zeta'_{i(\alpha)})$ for $Q'$ by $\zeta'_{i(\alpha)} = \zeta_i$ for all $i \in I$ and $1 \le \alpha \le v_i$. Note that $\zeta$ restricts to $\zeta'$ on $T_Q(\bv)$, so if $\bM_{Q'}(\bv',\bw')^{\zeta'\text{\rm{-ss}}} / G_{Q'}(\bv')$ is Deligne--Mumford, then the same is true for $(\bM_{Q}(\bv,\bw) \times \fg_Q(\bv))^{\zeta\text{\rm{-ss}}} / T_Q(\bv)$.

Now, given $\theta' = (\theta'_i)_{i \in I'}$, define $\theta = (\theta_i)_{i \in I}$ by $\theta_i = \theta'_{i(1)} + \cdots + \theta'_{i(v_i)}$. It follows from the definitions that $\theta' \cdot \zeta' = \theta \cdot \zeta$ and that $\theta_i \le v_i$ for all $i \in I$ if $\theta'_i \le v'_i = 1$ for all $i \in I'$. In particular, we have assumed that $\theta \cdot \zeta \ne 0$ whenever $\theta \ne 0$ and $\theta_i \le v_i$ for all $i \in I$. This implies that $\theta' \cdot \zeta' \ne 0$ for all $\theta' \ne 0$ with $\theta'_i \le v'_i$ for all $i \in I'$. In the proof of \cite[Lemma 2.12]{nakajima}, it is shown\footnote{The condition that the point is in $\mu^{-1}(0)$ in \cite{nakajima} is only used to conclude that $\theta'$ is a ``root''. But we are not requiring this latter condition, so we do not need to impose the condition $\mu=0$.} that if there is a point which is $\zeta'$-semistable but not $\zeta'$-stable, then there exists $\theta' \ne 0$ with $\theta'_i \le v'_i$ for all $i \in I$ such that $\theta' \cdot \zeta' = 0$. So we conclude that $\zeta'$-semistable points in $\bM_{Q'}({\bf v}', {\bf w}')$ are also $\zeta'$-stable. 

So the condition we have assumed implies that $(\bM_Q(\bv,\bw) \times \fg_Q(\bv))^{\zeta\text{\rm{-ss}}} / T_Q(\bv)$ is Deligne--Mumford. By \autoref{prop:semistable} and the first claim, this is equivalent to $\zeta$ being generic for $\ol{\Sigma}$ (note that $\bM_Q(\bv,\bw) \times \fg_Q(\bv)$ is self-dual, and hence quasi-symmetric), and this proves the second claim. The set of $\zeta$ satisfying the condition in the second claim is in the complement of finitely many hyperplanes in $M_\bR^W$, and hence $(M_\bR^W)_{\ol{\Sigma}-{\rm gen}} \ne \emptyset$.
\end{proof}

\section{Action of the fundamental groupoid of the complexified K\"ahler moduli space}

In this section, we use magic windows to show that the derived equivalences of \autoref{cor:equivalences} and \autoref{thm:hyperkaehler} fit together to form a representation of the fundamental group of a space which is a mathematical interpretation of the ``complexified K\"{a}hler moduli space'' studied in $N=2$ superconformal field theory. In particular, we generalize and make precise the physical picture of \cite{HHP}. As in the rest of the paper, we will fix a quasi-symmetric linear representation $X$ of a split reductive group $G$. We consider the $M$-periodic locally finite hyperplane arrangement $\{H_\alpha\subset M_\bR^W\}$ of \autoref{lem:hyperplane}, which is defined so that $\partial(\delta + \ol{\nabla}) \cap M = \emptyset$ for any $\delta \in M_\bR^W$ in the complement of this hyperplane arrangement. We define
\[
\cK_X := \left( M_\bC^W \setminus {\bigcup}_\alpha H_\alpha \otimes \bC \right) / M^W.
\]
For the point $\delta + \ell i \in M_\bC^W$, in order to match our notation with the physics literature, one identifies $\delta$ with the ``$\Theta$ parameters'' and $\ell$ with the ``FI parameters'' of \cite[\S 4.3]{HHP}.

After establishing a convenient combinatorial presentation for the fundamental groupoid $\Pi_1(\cK_X)$ in \autoref{prop:groupoid}, we will establish a categorical representation of this groupoid. To make this precise, we let $\op{Ho}(\dgCat)$ denote the category whose objects are dg-categories and whose morphisms are quasi-isomorphism classes of dg-functors between dg-categories. The main result of this section, \autoref{prop:representation} constructs a functor $F \colon \Pi_1(\cK_X) \to \op{Ho}(\dgCat)$ such that any point in $\cK_X$ is mapped to a category which is canonically identified with ${\rm D}^b(X^{\rm{ss}}(\ell)/G)$ for any $\ell$ such that $X^{\rm{ss}}(\ell) = X^{\rm{s}}(\ell)$.

\subsection{A presentation for the fundamental groupoid of the complement of a hyperplane arrangement}

Let $V$ be a real vector space along with a locally finite hyperplane arrangement $\cA = \{H_\alpha \subset V\}$. Fix a subset $V_{{\rm gen}} \subset V$ with the following property:
\begin{itemize}
\item[] For any finite collection of hyperplanes $H_{\alpha_0},\ldots,H_{\alpha_n} \in \cA$, we let $H'_{\alpha_0},\ldots,H'_{\alpha_n}$ be the same hyperplanes translated so that they pass through the origin. Then there is some $\ell \in V_{{\rm gen}}$ lying in each connected component of $V \setminus (H'_{\alpha_0} \cup \cdots \cup H'_{\alpha_n})$.
\end{itemize}

For simplicity we will also assume that $V_{{\rm gen}}$ is symmetric with respect to the origin, so $\ell \in V_{{\rm gen}}$ if and only if $-\ell \in V_{{\rm gen}}$, although this is not strictly necessary for the result below. Consider the following directed graph:
\begin{itemize}
\item \emph{Vertices}: one vertex $[\delta]$ for each $\delta \in V$ which does not lie on a hyperplane, and
\item \emph{Edges}: one edge $[\delta] \xrightarrow{\ell} [\delta']$ labeled by an element $\ell \in V_{{\rm gen}}$ whenever $\ell$ is not parallel to any of the hyperplanes meeting the line segment $\{(1-t)\delta + t \delta' \mid t\in [0,1]\}$.
\end{itemize}

In the example of interest in this paper, $V = M_\bR^W$ and $V_{{\rm gen}} = \{\ell \mid X^{\rm ss}(\ell) = X^{\rm s}(\ell)\}$.

\begin{defn} \label{defn:groupoid}
We define the groupoid $\Gamma(V,\cA)$ to be the free groupoid on this directed graph modulo the congruence generated by the following equivalences of arrows
\begin{enumerate}
\item $[\delta] \xrightarrow{\ell} [\delta]$ is congruent to the identity morphism for any $\ell$,
\item $[\delta] \xrightarrow{\ell} [\delta'']$ is congruent to the composition of $[\delta] \xrightarrow{\ell} [\delta']$ and $[\delta'] \xrightarrow{\ell} [\delta'']$ whenever all three are valid arrows,\footnote{Note that any hyperplane which meets the convex hull, $\op{Hull}(\delta,\delta',\delta'')$, must also meet one of the boundary line segments, so the condition that all three are valid arrows is equivalent to all three vertices lying in the complement of the hyperplanes, and $\ell$ is not parallel to any hyperplane meeting $\op{Hull}(\delta,\delta',\delta'')$.} and
\item $[\delta] \xrightarrow{\ell} [\delta']$ is congruent to $[\delta] \xrightarrow{\ell'} [\delta']$ as long as $\ell$ and $\ell'$ have the same orientation relative to each hyperplane meeting the line segment $\{(1-t)\delta + t \delta' \mid t\in [0,1]\}$.
\end{enumerate}
\end{defn}

Our goal is to identify $\Gamma(V,\cA)$ with the fundamental groupoid $\Pi_1$ of the topological space $(V \otimes_\bR \bC) \setminus \bigcup_\alpha (H_\alpha \otimes_\bR \bC)$. For every object $[\delta] \in \Gamma(V,\cA)$, we assign the point $\delta + 0 i \in V \otimes_\bR \bC$, and for every morphism $[\delta] \xrightarrow{\ell} [\delta']$ we assign the homotopy class of the path
\[
\gamma(t) = \begin{cases} \delta + 3 t \ell i , & t\in [0,1/3] \\ (2-3t)\delta + (3t-1) \delta' + \ell i, & t \in [1/3,2/3] \\ \delta' + (3-3 t) \ell i , & t\in [2/3,1] \end{cases}
\]
from $\delta + 0i$ to $\delta'+0i$. It is straightforward to show that each of the equivalences of arrows in (1), (2), and (3) correspond to a homotopy of paths, so this assignment defines a functor
\[
\Phi \colon \Gamma(V,\cA) \to \Pi_1(V\otimes_\bR \bC \setminus {\bigcup}_\alpha H_\alpha \otimes_\bR \bC).
\]

\begin{prop} \label{prop:groupoid-equiv}
The functor $\Phi$ is an equivalence of groupoids.
\end{prop}

We prove this by comparing with the presentation of the fundamental groupoid given in \cite[\S 2]{paris}. Before doing this, we set up some preparatory results. First, choose a subset $D \subset V$ such that each connected component of $V \setminus \bigcup_\alpha H_\alpha$, which we refer to as a cell, contains exactly one point $\delta \in D$. The cells of $V$ are in bijection with elements of $D$, and $\Gamma(V,\cA)$ is equivalent to the full subcategory $\cC \subset \Gamma(V,\cA)$ having objects $[\delta]$ with $\delta \in D$.

Fix $\ell_0 \in V_{{\rm gen}}$ which is not parallel to any of the hyperplanes in $\cA$. Let $P$ denote the set of arrows $[\delta] \xrightarrow{\pm \ell_0} [\delta']$ such that $\delta$ and $\delta'$ lie in adjacent cells of $V$, i.e., they are separated by a single hyperplane, and the sign of $\pm \ell_0$ is chosen so that the label of the arrow has the same orientation as the vector $\delta' - \delta$ relative to the unique hyperplane separating $\delta$ and $\delta'$. Note that we have thus chosen a single object for each cell and a single arrow between any pair of adjacent objects and hence we have defined a directed graph with vertex set $D$. We let $\op{Free}(D,P)$ denote the free groupoid on this directed graph, and consider the tautological functor $\op{Free}(D,P) \to \Gamma(V,\cA)$.

Define a path to be a composition of arrows in this directed graph together with their formal inverses. It is positive if it is a composition of arrows from the directed graph constructed above (as opposed to taking formal inverses); the length of a path is the number of arrows it uses. A positive path is minimal if there is no shorter positive path with the same starting and ending point. Hence, given $\delta, \delta' \in D$, we can define the distance $d(\delta,\delta')$ to be the length of a minimal path starting at $\delta$ and ending at $\delta'$. If we let $h(\delta, \delta')$ denote the number of hyperplanes $H$ such that $\delta$ and $\delta'$ have opposite orientations with respect to $H$, then a simple argument by induction on $h$ shows that $d(\delta,\delta') = h(\delta,\delta')$.

\begin{lem} \label{lem:minimal-path-relation}
Pick a minimal positive path $[\delta_1] \xrightarrow{\ell_1} [\delta_2] \xrightarrow{\ell_2} \cdots \xrightarrow{\ell_n} [\delta_{n+1}]$ where each $\ell_i$ is either $\ell_0$ or $-\ell_0$. For all $1 \le i \le n$, $\delta_{n+1}-\delta_1$ has the same orientation as $\ell_i$ relative to the hyperplane that separates the cells containing $\delta_i$ and $\delta_{i+1}$.
\end{lem}

\begin{proof}
Pick $i$ and let $H$ be the hyperplane separating $\delta_i$ and $\delta_{i+1}$. Since $d(\delta_1,\delta_{n+1}) = h(\delta_1,\delta_{n+1})$, the minimal positive path cannot cross a hyperplane more than once, and in particular, cannot cross $H$ more than once. So the orientations of $\delta_1, \delta_2, \dots, \delta_i$ are all the same relative to $H$, and the orientations of $\delta_{i+1}, \delta_{i+2}, \dots, \delta_{n+1}$ are all the same relative to $H$ (but different from the first set). In particular, $\delta_{n+1} - \delta_1$ has the same orientation as $\delta_{i+1} - \delta_i$ relative to $H$, and by definition, this has the same orientation as $\ell_i$.
\end{proof}

\begin{proof}[Proof of \autoref{prop:groupoid-equiv}]
It follows from \autoref{defn:groupoid} that the functor $\op{Free}(D,P) \to \Gamma(V,\cA)$ is essentially surjective and surjective on Hom sets: by (1) and (2) every object is isomorphic to some $[\delta]$ with $\delta \in D$, by (2) every arrow labeled by $\ell$ can be decomposed into a composition of arrows labeled by $\ell$ between adjacent cells, and by (3) each of these arrows is equivalent to the same arrow labeled by $\pm \ell_0$ (or its inverse). Furthermore, by \cite[Theorem 2.1]{paris}\footnote{This reference assumes that there are finitely many hyperplanes, but it is easy to deduce the locally finite situation from this.} (see also \cite{salvetti}), the composition
\[
\op{Free}(D,P) \to \Gamma(V,\cA) \xrightarrow{\Phi} \Pi_1(V \otimes_\bR \bC \setminus {\bigcup}_\alpha H_\alpha \otimes_\bR \bC)
\]
is essentially surjective and surjective on Hom sets as well, and the Hom sets of the fundamental groupoid are quotients of the arrows in $\op{Free}(D,P)$ by the smallest congruence which identifies minimal positive paths in $\op{Free}(D,P)$ whenever they have the same beginning and ending point. Therefore, in order to show that $\Phi$ is an equivalence, it suffices to show that any two minimal positive paths in $\op{Free}(D,P)$ are identified in $\Gamma(V,\cA)$. Given a minimal positive path $[\delta_1] \xrightarrow{\ell_1} \cdots \xrightarrow{\ell_n} [\delta_{n+1}]$ where each $\ell_i$ is either $\ell_0$ or $-\ell_0$, \autoref{lem:minimal-path-relation} together with (3) implies that we can replace all of the arrow labels by $\delta_{n+1} - \delta_1$. Now (2) implies that this composition is the same as $[\delta_1] \xrightarrow{\delta_{n+1} - \delta_1} [\delta_{n+1}]$, so any two minimal positive paths are identified.
\end{proof}

\subsubsection{Equivariance with respect to translations} 
Next let us assume that $V = L \otimes_\bZ \bR$ is the real vector space generated by a lattice $L$, and that the locally finite hyperplane arrangement $\{H_\alpha \subset L_\bR\}$ is invariant under the action of $L$ by translation.

\begin{defn} \label{defn:larger_groupoid}
Let $\tilde{\Gamma}(L,\cA)$ be the enlargement of the groupoid $\Gamma(L_\bR,\cA)$ defined above which is obtained by formally adding an arrow $[\delta]  \leadsto [\delta + m]$ for every $\delta$ in the complement of the $H_\alpha$ and every $m \in L$. We impose the additional relations
\begin{enumerate}
\item arrows of the form $\leadsto$ commute with those of the form $\xrightarrow{\ell}$ in the sense that 
\[
([\delta] \xrightarrow{\ell} [\delta'] \leadsto [\delta' + m]) \sim ([\delta] \leadsto [\delta+m] \xrightarrow{\ell} [\delta' + m]) 
\] 
whenever the arrow $\xrightarrow{\ell}$ is valid, and
\item the arrows $[\delta] \leadsto [\delta + m]$ and $[\delta + m] \leadsto [\delta + m + m']$ compose to $[\delta] \leadsto [\delta + m + m']$.
\end{enumerate}
\end{defn}

The new morphisms $[\delta] \leadsto [\delta + m]$ correspond to the deck transformations of the covering map $L_\bC \to L_\bC / L$. 

\begin{prop} \label{prop:groupoid}
The isomorphism $\Phi \colon \Gamma(L_\bR,\cA) \to \Pi_1(L_\bC \setminus \bigcup_\alpha H_\alpha \otimes \bC)$ extends to an isomorphism
\[
\tilde{\Phi} \colon \tilde{\Gamma}(L,\cA) \to \Pi_1\left( \left(L_\bC \setminus {\bigcup}_\alpha H_\alpha \otimes \bC \right) / L \right)
\]
compatible with the canonical faithful functor $\Pi_1(L_\bC \setminus \bigcup_\alpha H_\alpha \otimes \bC) \subset \Pi_1\left( \left(L_\bC \setminus \bigcup_\alpha H_\alpha \otimes \bC \right) / L \right)$.
\end{prop}
\begin{proof}
This is a straightforward application of the description of the fundamental groupoid of a quotient of a Hausdorff space by a discontinuous group action \cite[Chapter 11]{BrownGroupoids}. In particular, the groupoid $\tilde{\Gamma}(L,\cA)$ we have defined is actually the semidirect product $\Gamma(L_\bR,\cA) \rtimes L$ for the action of the group $L$ on $\Gamma(L_\bR,\cA)$ by translation (which is compatible with the isomorphism $\Phi$). By \cite[11.2.1]{BrownGroupoids}, the fundamental groupoid of the quotient space is canonically identified with the orbit groupoid $\Gamma(L_\bR,\cA) \git L$, which is canonically identified with $(\Gamma(L_\bR,\cA) \rtimes L) / N$ in \cite[11.5.1]{BrownGroupoids}, where $N \subset \Gamma(L_\bR,\cA) \rtimes L$ is the normal subgroupoid consisting of arrows of the form $[\delta] \leadsto [\delta + m]$. Finally, $N$ is contractible, because $L$ acts freely on $\Gamma(L_\bR,\cA)$, so the quotient morphism $\Gamma(L_\bR,\cA) \rtimes L \to (\Gamma(L_\bR,\cA) \rtimes L) / N$ is an equivalence as well.
\end{proof}

\subsection{A categorical representation of the groupoid $\tilde{\Gamma}(M^W,\cA)$}

We now return to the specific hyperplane arrangement of interest in this paper, where $X$ is a quasi-symmetric linear representation of a split reductive group $G$. We assume that $\ol{\Sigma}$ generates $M_\bR$ and that $(M_\bR^W)_{\ol{\Sigma}\text{\rm{-gen}}} \neq \emptyset$, so \autoref{thm:magic_windows} implies that for $\delta \in M_\bR^W$ outside of a certain locally finite $M^W$-periodic hyperplane arrangement $\cA := \{H_\alpha \subset M_\bR^W\}$ (described in \autoref{lem:hyperplane}) the restriction functor is an equivalence $\op{res}_{X^{\rm{ss}}(\ell)} \colon \cM(\delta + \ol{\nabla}) \to {\rm D}^b(X^{\rm{ss}}(\ell)/G)$ when $X^{\rm{ss}}(\ell) = X^{\rm{s}}(\ell)$.

\begin{prop} \label{prop:representation}
Assigning $F([\delta]):=\cM(\delta + \ol{\nabla})$, assigning $F([\delta] \xrightarrow{\ell} [\delta'])$ to the functor
\[
\cM(\delta + \ol{\nabla}) \xrightarrow{\op{res}_{X^{\rm ss}(\ell)}} {\rm D}^b(X^{\rm{ss}}(\ell)/G) \xrightarrow{\op{res}_{X^{\rm ss}(\ell)}^{-1}} \cM(\delta' + \ol{\nabla}),
\]
and assigning $F([\delta] \leadsto [\delta + m])$ to the functor
\[
\cM(\delta + \ol{\nabla}) \xrightarrow{(\cO_X \otimes m) \otimes (-)} \cM(\delta+m+\ol{\nabla})
\]
extends uniquely to a representation of the groupoid $F \colon \tilde{\Gamma}(M^W,\cA) \to \op{Ho}(\dgCat)$.
\end{prop}

In particular, combining this with \autoref{prop:groupoid} gives a categorical representation of $\Pi_1(\cK_X)$ such that each point is mapped to a category which is canonically isomorphic to ${\rm D}^b(X^{\rm{ss}}(\ell)/G)$ under restriction. Most of the proof of this proposition is formal. The only categorical input is the following:

\begin{lem} \label{lem:adjoints}
Let $\ell \in M_\bR^W$ be such that $X^{\rm{ss}}(\ell) = X^{\rm{s}}(\ell)$, and let $\delta \in M_\bR^W$ be such that $\partial(\delta + \ol{\nabla}) \cap M \neq \emptyset$ but for $0 < t \ll 1$ we have $\partial(\delta \pm t \ell + \ol{\nabla}) \cap M = \emptyset$. Then the composition
\[
\cM(\delta + \ol{\nabla}) \xrightarrow{\op{res}_{X^{\rm{ss}}(\ell)}} {\rm D}^b(X^{\rm{ss}}(\ell)/G) \xrightarrow{\op{res}_{X^{\rm{ss}}(\ell)}^{-1}} \cM(\delta + t \ell + \ol{\nabla})
\]
is the right adjoint of the inclusion $\cM(\delta+t \ell + \ol{\nabla}) \subset \cM(\delta + \ol{\nabla})$. Likewise, the composition
\[
\cM(\delta + \ol{\nabla}) \xrightarrow{\op{res}_{X^{\rm{ss}}(\ell)}} {\rm D}^b(X^{\rm{ss}}(\ell)/G) \xrightarrow{\op{res}_{X^{\rm{ss}}(\ell)}^{-1}} \cM(\delta - t \ell + \ol{\nabla})
\]
is the left adjoint of the inclusion $\cM(\delta - t \ell + \ol{\nabla}) \subset \cM(\delta + \ol{\nabla})$.
\end{lem}

\begin{proof}
Let $F \in \cM(\delta + t \ell + \ol{\nabla})$ for $0 < t \ll 1$, or equivalently $F \in \cM(\delta + \nabla_\ell)$. Then the first claim of the lemma amounts to the claim that
\begin{equation} \label{eqn:restriction}
\RHom_{X/G} (F,G) \simeq \RHom_{X^{\rm{ss}}(\ell)/G} (F|_{X^{\rm{ss}}(\ell)}, G|_{X^{\rm{ss}}(\ell)})
\end{equation}
for all $G \in \cM(\delta + \ol{\nabla})$. As in \S \ref{sect:magic_windows} we use the stratification of the unstable locus in GIT associated to the linearization $\ell$, which specifies a sequence of distinguished one-parameter subgroups $\lambda_0, \ldots, \lambda_n$ with $\pair{\lambda_i}{\ell}>0$ and fixed loci $Z^{\rm ss}_i \subset X^{\lambda_i}$. Observe from the proof of \autoref{lem:fully_faithful} and specifically \autoref{eqn:weight_bounds}, we have
\[
\op{wt}_{\lambda_i} (F|_{Z_i}) \in \pair{\lambda_i}{\delta} + (-\eta_{\lambda_i}/2,\eta_{\lambda_i}/2]
\]
for any of these distinguished $\lambda_i$. In particular if we choose a very small constant $0<a\ll 1$ and let $w = (w_0,\ldots,w_n)$ with $w_i = \pair{\lambda_i}{\delta} - \eta_{\lambda_i}/2 + a$, then in the notation of \cite{HL} we have $F \in {\rm D}^b(X/G)_{\geq w}$ and similarly $G \in {\rm D}^b(X/G)_{< w}$, so the generalized ``quantization commutes with reduction'' theorem of \cite[Theorem 3.26]{HL} implies that the restriction map \eqref{eqn:restriction} is an equivalence.

The second claim amounts to showing the same equivalence \eqref{eqn:restriction}, but this time when $F \in \cM(\delta + \ol{\nabla})$ and $G \in \cM(\delta + \nabla_{-\ell})$. This time \autoref{eqn:weight_bounds} implies that
\[
\op{wt}_{\lambda_i} (G|_{Z_i}) \in \pair{\lambda_i}{\delta} + [-\eta_{\lambda_i}/2,\eta_{\lambda_i}/2).
\]
So choosing $w_i = \pair{\lambda_i}{\delta} - \eta_{\lambda_i}/2$, we have $F \in {\rm D}^b(X/G)_{\geq w}$ and $G \in {\rm D}^b(X/G)_{<w}$, so again the quantization commutes with reduction theorem implies that \eqref{eqn:restriction} is an equivalence.
\end{proof}

\begin{proof}[Proof of \autoref{prop:representation}]
Since $\Gamma(V,\cA)$ is defined as a free groupoid modulo some congruence, it suffices to verify that the arrows described in (1), (2), and (3) of \autoref{defn:groupoid} give isomorphic functors. (1) is immediate, because when $\delta = \delta'$, the corresponding functor $\cM(\delta +\ol{\nabla}) \to \cM(\delta + \ol{\nabla})$ is the identity functor regardless of what $\ell$ is. The congruence (2) follows since the following diagram commutes up to isomorphism of functors
\[
\begin{gathered}
\xymatrix{\cM(\delta + \ol{\nabla}) \ar[dr]^{\op{res}_{X^{\rm{ss}}(\ell)}} \ar[rr] & & \cM(\delta' + \ol{\nabla})  \ar[dr]^{\op{res}_{X^{\rm{ss}}(\ell)}} \ar[rr] & & \cM(\delta'' + \ol{\nabla}) \\ & {\rm D}^b(X^{\rm{ss}}(\ell)/G) \ar[ur]^{\op{res}_{X^{\rm{ss}}(\ell)}^{-1}} \ar[rr]^{\op{id}} & & {\rm D}^b(X^{\rm{ss}}(\ell)/G) \ar[ur]^{\op{res}_{X^{\rm{ss}}(\ell)}^{-1}} } 
\end{gathered}.
\]

So the only thing which really requires a verification is the congruence (3). First observe that the category $\cM(\delta + \ol{\nabla})$ does not change (as a subcategory of ${\rm D}^b(X/G)$) if $\delta$ lies in the interior of the hyperplane arrangement and we perturb $\delta$ slightly. Likewise, the functor $F([\delta] \xrightarrow{\ell} [\delta'])$ is insensitive to small perturbations of $\delta$ and $\delta'$. We may therefore assume that the line segment $\{(1-t) \delta + t \delta' \mid t\in [0,1]\}$ meets the hyperplane arrangement generically, i.e., that no two hyperplanes meet the segment at the same value of $t$. Using the congruence (2), we can therefore factor this arrow as a composition of a sequence of arrows, each of which meets at most one hyperplane, so it suffices to verify (3) in this case.

When no hyperplanes separate $\delta$ and $\delta'$, then the equivalence is quasi-isomorphic to the identity functor $\cM(\delta+\ol{\nabla}) = \cM(\delta'+\ol{\nabla})$, and does not depend on $\ell$. So assume that the line segment between $\delta$ and $\delta'$ meets a unique hyperplane at the point $\delta_0$, and let $\kappa = 1$ if $\ell$ has the same orientation as $\delta'-\delta$ relative to this hyperplane and $\kappa = -1$ otherwise. Then we have
\[
\cM(\delta + \ol{\nabla}) = \cM(\delta_0 - \kappa t \ell + \ol{\nabla}) \quad \text{ and } \quad \cM(\delta' + \ol{\nabla}) = \cM(\delta_0 + \kappa t\ell + \ol{\nabla})
\] 
for $ 0<t\ll 1$. \autoref{lem:adjoints} implies that the equivalence $F([\delta] \xrightarrow{\ell} [\delta'])$ is just the composition of the inclusion $\cM(\delta+\ol{\nabla}) \subset \cM(\delta_0 + \ol{\nabla})$ followed by the right (respectively, left) adjoint of the inclusion $\cM(\delta' + \ol{\nabla}) \subset \cM(\delta_0 + \ol{\nabla})$ when $\kappa = 1$ (respectively, $\kappa = -1$). This shows that $F([\delta] \xrightarrow{\ell} [\delta'])$ only depends on the orientation of $\delta$ relative to the hyperplane separating $\delta$ and $\delta'$ in this case, which verifies (3).

Finally, in order to show that this in fact extends from $\Gamma(M_\bR^W,\cA)$ to $\tilde{\Gamma}(M^W,\cA)$, it suffices to show that the two functors are equivalent
\[
F([\delta] \xrightarrow{\ell} [\delta'] \leadsto [\delta' + m]) \simeq F([\delta] \leadsto [\delta+m] \xrightarrow{\ell} [\delta' + m]).
\]
This follows from commutativity up to equivalence of the following diagram
\[
\begin{gathered}
\xymatrixcolsep{5pc}\xymatrix{\cM(\delta+\ol{\nabla}) \ar[d]_{(-)\otimes (\cO_X \otimes m)} \ar[r]^-{\op{res}_{X^{\rm{ss}}(\ell)}} & {\rm D}^b(X^{\rm{ss}}(\ell) / G) \ar[d]^{(-)\otimes (\cO_X \otimes m)|_{X^{\rm{ss}}(\ell)}} & \cM(\delta'+\ol{\nabla}) \ar[l]_-{\op{res}_{X^{\rm{ss}}(\ell)}} \ar[d]^{(-)\otimes (\cO_X \otimes m)} \\ 
\cM(\delta+\ol{\nabla}) \ar[r]^-{\op{res}_{X^{\rm{ss}}(\ell)}} & {\rm D}^b(X^{\rm{ss}}(\ell) / G) & \cM(\delta'+\ol{\nabla}) \ar[l]_-{\op{res}_{X^{\rm{ss}}(\ell)}}}
\end{gathered}. \qedhere
\]
\end{proof}

\begin{rem}
These actions of $\Pi_1(\cK_X)$ on ${\rm D}^b(X^{\rm{ss}}(\ell)/G)$ generalize the example studied in \cite{DonovanSegal2}, which constructed a mixed braid group action on the derived category of certain deformations of the minimal resolution of the ${\rm A}_k$ surface singularity. There they realized these deformations as the GIT quotient of a certain self-dual linear representation of a torus, and they used magic windows to deduce the braid relations. In fact, Donovan and Segal show that the representation of the mixed braid group is faithful, and it is an interesting question as to when the larger class of examples produced by our methods lead to faithful representations.
\end{rem}

\begin{rem} \label{rmk:K-theory}
The action of $\Pi_1(\cK_X)$ on ${\rm D}^b(X^{\rm ss}(\ell)/G)$ induces an action on its $K$-theory. Assume that $G$ is connected, so that we may index irreducible representations $V(\chi)$ by their highest weight $\chi$. Choose $\delta \in M_\bR^W$ in the complement of the hyperplane arrangement $\cA$ determined by $\ol{\nabla}$ (defined in \autoref{lem:hyperplane}). The restriction map $\cM(\delta + \ol{\nabla}) \to {\rm D}^b(X^{\rm ss}(\ell)/G)$ is then an equivalence, so by \autoref{prop:K_theory}, $K_0(\cM(\delta + \ol{\nabla}))$ has a natural basis given by the classes
\[
\cV(\chi) := [\cO_X \otimes V(\chi)] \in K_0(\cM(\delta + \ol{\nabla}))
\]
for those $\chi \in M^+ \cap (-\rho + \delta + \frac{1}{2} \ol{\Sigma})$. Our goal is to explicitly describe the action of the equivalence
\[
\Phi := K_0(F([\delta] \xrightarrow{\ell} [\delta'])) \colon K_0 (\cM(\delta +\ol{\nabla})) \to K_0(\cM(\delta'+\ol{\nabla}))
\]
in the respective bases, when $\delta'\in M_\bR^W$ is separated from $\delta$ by exactly one hyperplane $H \in \cA$.

Let $\delta_0$ be the point on the line segment joining $\delta$ and $\delta'$ which lies on the hyperplane $H$ separating $\delta$ and $\delta'$, and choose an $\varepsilon \in M_\bR^W$ which is very small, generic for $\ol{\Sigma}$, and oriented so that $\delta_0 + \varepsilon$ lies on the same side of $H$ as $\delta'$. Then we have
\[
\cM(\delta + \ol{\nabla}) = \cM(\delta_0 - t \varepsilon + \ol{\nabla}) \quad \text{and} \quad \cM(\delta'+\ol{\nabla}) = \cM(\delta_0 + t \varepsilon + \ol{\nabla})
\]
for $0<t\ll 1$. In particular, the basis of $K_0(\cM(\delta + \ol{\nabla}))$ consists of $\cV(\chi)$ with $\chi  \in (-\rho + \delta_0 + \frac{1}{2} \ol{\Sigma}_{- \varepsilon}) \cap M^+$ and the basis of $K_0(\cM(\delta' + \ol{\nabla}))$ consists of $\cV(\chi)$ with $\chi \in (-\rho + \delta_0 + \frac{1}{2} \ol{\Sigma}_{\varepsilon}) \cap M^+$. For each $\chi \in (-\rho + \delta_0 + \frac{1}{2} \ol{\Sigma}_{-\varepsilon}) \cap M^+$, we wish to express $\Phi(\cV(\chi))$ in the corresponding basis for $K_0(\cM(\delta'+\ol{\nabla}))$. There are two cases:
\begin{enumerate}[(1)]
\item If $\chi \in -\rho + \delta_0 + \frac{1}{2} \ol{\Sigma}_{\varepsilon}$ as well, which happens if and only if $\chi$ lies in the interior of $-\rho + \delta_0 + \frac{1}{2} \ol{\Sigma}$, then $\Phi(\cV(\chi)) = \cV(\chi)$.

\item Otherwise, following \S\ref{sec:SVDB}, there is an anti-dominant $\lambda$ which is constant on a facet of $\ol{\Sigma}$ such that $\pair{\lambda}{\chi} \le \pair{\lambda}{\mu}$ for all $\mu \in -\rho + \delta_0 +\varepsilon + r_\chi \ol{\Sigma}$.
\begin{itemize}
\item If $\pair{\lambda}{\ell} < 0$, then we use the class $[D^\vee_{\lambda,\chi}] \in K_0({\rm D}^b(X/G))$, whose restriction to $X^{\rm{ss}}(\ell)/G$ vanishes by \autoref{prop:D_unstable}. The class of $[D^\vee_{\lambda,\chi}]$ can be written as a sum (see \autoref{rmk:euler-char}) whose first term is $\cV(\chi)$ and whose remaining terms are multiples of $\cV(\nu)$ for $\nu$ for which $(r_\nu,p_\nu)$ is lexicographically smaller than $(r_\chi,p_\chi)$ in the notation of \S\ref{sec:SVDB} with respect to $-\rho + \delta_0+ \varepsilon + \frac{1}{2}\ol{\Sigma}$. Thus we can obtain an expression for $\cV(\chi)|_{X^{\rm{ss}}(\ell)}$ as a linear combination of these $\cV(\nu)|_{X^{\rm{ss}}(\ell)/G}$.
\item If $\pair{\lambda}{\ell} > 0$, then by \autoref{prop:C_unstable} the restriction of $[C_{\lambda,\chi}]$ to $X^{\rm{ss}}(\ell) / G$ vanishes, and again we can derive an expression for $\cV(\chi)|_{X^{\rm{ss}}(\ell)}$ as a linear combination of $\cV(\nu)|_{X^{\rm{ss}}(\ell)}$ where $(r_\nu,p_\nu)$ is smaller than $(r_\chi,p_\chi)$.
\end{itemize}
\end{enumerate}
We can repeat this process until we have expressed each $\cV(\chi)|_{X^{\rm{ss}}(\ell)}$ with $\chi \in (-\rho +\delta_0 + \frac{1}{2} \ol{\Sigma}) \cap M^+$ as an integer linear combination of $\cV(\nu)|_{X^{\rm{ss}}(\ell)}$ with $\nu \in (-\rho +\delta_0 + \frac{1}{2} \ol{\Sigma}_{\varepsilon}) \cap M^+$. The resulting linear combination is unique, and is precisely the column of the matrix for the equivalence $\Phi$ corresponding to the basis vector $\cV(\chi)$ for $\chi \in -\rho+\delta_0+\frac{1}{2} \ol{\Sigma}_{-\varepsilon}$. Note that there might be $\chi \in -\rho + \delta_0 + \frac{1}{2} \ol{\Sigma}$ which lie in neither $-\rho + \delta_0 + \frac{1}{2} \ol{\Sigma}_\varepsilon$ nor $-\rho + \delta_0 + \frac{1}{2} \ol{\Sigma}_{-\varepsilon}$, but it might still be necessary to compute $\cV(\chi)|_{X^{\rm{ss}}(\ell)}$, as the following example shows:

\begin{ex}
Continuing the example studied in \autoref{fig:sym3}, we consider $G = {\rm GL}_2(k)$, $X = T^* {\rm Sym}^3(k^2)$ and $\delta_0 = (\frac{1}{2}, \frac{1}{2})$, $\ell=(1,1)$, and $\varepsilon = (\frac{1}{4},\frac{1}{4})$. There are four dominant weights $\chi \in (-\rho + \delta_0 + \frac{1}{2} \ol{\Sigma}) \cap M^+$ which do not lie in $-\rho + \delta_0 + \frac{1}{2} \ol{\Sigma}_{\varepsilon}$. We illustrate the procedure for computing $\Phi(\cV(\chi))$ for two of these $\chi$ in \autoref{fig:changebasis}. Ultimately we compute (note all of the cancellation):
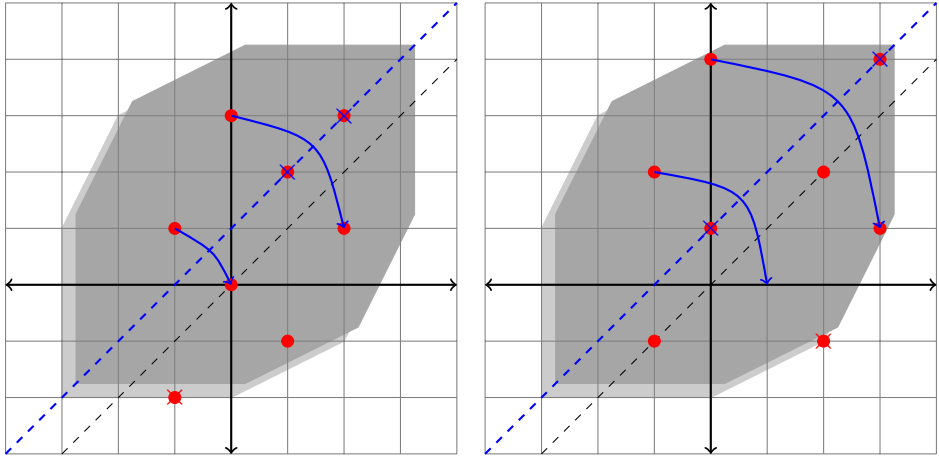
\begin{figure}
\begin{tabular}{ll}
\begin{tikzpicture}[scale=.75]
\filldraw[black!20!white] (-3,1) -- (-2,3) -- (0,4) -- (3,4) -- (3,1) -- (2,-1) -- (0,-2) -- (-3,-2) -- cycle;
\filldraw[black!35!white] (-3+.25,1+.25) -- (-2+.25,3+.25) -- (0+.25,4+.25) -- (3+.25,4+.25) -- (3+.25,1+.25) -- (2+.25,-1+.25) -- (0+.25,-2+.25) -- (-3+.25,-2+.25) -- cycle;
\draw[step=1, color=gray] (-4,-3) grid (4,5);
\draw[<->, thick] (-4,0) -- (4,0);
\draw[<->, thick] (0,-3) -- (0,5);
\draw[dashed] (-3,-3) -- (4,4);
\draw[dashed, thick, color=blue] (-4,-3) -- (4,5);
\foreach \i in {(2,3), (1,2), (0,3), (-1,1), (2,1), (0,0), (1,-1), (-1,-2)}
{ \filldraw[red] \i circle (3pt); }
\node[red] at (-1,-2) {$\times$};
\draw[->,thick, color=blue] (0,3) .. controls (1.6,2.6) .. (2,1);
\draw[->,thick, color=blue] (-1,1) .. controls (-.3,.6) .. (0,0);
\node[blue] at (1,2) {$\times$};
\node[blue] at (2,3) {$\times$};
\end{tikzpicture}
&
\begin{tikzpicture}[scale=.75]
\filldraw[black!20!white] (-3,1) -- (-2,3) -- (0,4) -- (3,4) -- (3,1) -- (2,-1) -- (0,-2) -- (-3,-2) -- cycle;
\filldraw[black!35!white] (-3+.25,1+.25) -- (-2+.25,3+.25) -- (0+.25,4+.25) -- (3+.25,4+.25) -- (3+.25,1+.25) -- (2+.25,-1+.25) -- (0+.25,-2+.25) -- (-3+.25,-2+.25) -- cycle;
\draw[step=1, color=gray] (-4,-3) grid (4,5);
\draw[<->, thick] (-4,0) -- (4,0);
\draw[<->, thick] (0,-3) -- (0,5);
\draw[dashed] (-3,-3) -- (4,4);
\draw[dashed, thick, color=blue] (-4,-3) -- (4,5);
\foreach \i in {(2,-1), (-1,-1), (3,1), (2,2), (0,1), (3,4), (-1,2), (0,4)}
{ \filldraw[red] \i circle (3pt); }
\node[red] at (2,-1) {$\times$};
\draw[->,thick, color=blue] (-1,2) .. controls (.7,1.7) .. (1,0);
\draw[->,thick, color=blue] (0,4) .. controls (2.5,3.5) .. (3,1);
\node[blue] at (0,1) {$\times$};
\node[blue] at (3,4) {$\times$};
\end{tikzpicture}
\end{tabular}
\caption{\footnotesize The blue dashed line is the line of reflection for the $\rho$-shifted action of the nontrivial element in $S_2$, and the curved blue arrows indicate reflection across that line. The lighter region is $-\rho + \delta_0 + \frac{1}{2}\ol{\Sigma}$ and the darker region is $-\rho + \delta_0 + \varepsilon + \frac{1}{2}\ol{\Sigma}$. In order to compute $\Phi(\cV(\chi))$ for $\chi \in (-\rho + \delta_0 -\varepsilon + \frac{1}{2}\ol{\Sigma}) \cap M^+$, we need to rewrite $\cV(\chi)|_{X^{\rm{ss}}(\ell)}$ for dominant weights in the lighter region in terms of $\cV(\nu)|_{X^{\rm{ss}}(\ell)}$ for dominant weights in the darker region. \emph{Left hand side}: For $\chi=(-1,-2)$ we have $\lambda = (0,1)$ and use the class $[C_{\lambda,\chi}]$ to rewrite $\cV(\chi)$. \emph{Right hand side}: For $\chi = (2,-1)$ we have $\lambda = (-1,2)$ and use the class $[C_{\lambda,\chi}]$ to rewrite $\cV(\chi)$.} \label{fig:changebasis}
\end{figure}
\begin{align*}
\Phi(\cV(i,j)) &= \cV(i,j) \quad \text{ for } (i,j) \in -\rho + \delta_0 + \frac{1}{2} \ol{\Sigma}^{\rm int},\\
\Phi(\cV(-2,-2)) &= \cV(2,0) - \cV(2,2),\\
\Phi(\cV(-1,-2)) &= \cV(1,-1), \\
\Phi(\cV(0,-2)) &= \cV(2,-1) + \cV(1,0) + \cV(0,0) - \cV(3,1) + \cV(3,3).
\end{align*}
We use the complex $[C_{\lambda,\chi}]$ with $\lambda = (0,1)$ to compute these last three expressions. Note, however, that $(2,-1) \notin -\rho + \delta_0  + \frac{1}{2}\ol{\Sigma}_{\varepsilon}$, hence $\cV(2,-1)$ is not in the final basis for $K_0(\cM(\delta'+\ol{\nabla}))$. We must use $[C_{\lambda,\chi}]$ with $\lambda = (-1,2)$ to rewrite
\[
\cV(2,-1)|_{X^{\rm{ss}}(\ell)} = \cV(-1,-1)|_{X^{\rm{ss}}(\ell)} + \cV(2,2)|_{X^{\rm{ss}}(\ell)} + \cV(1,0)|_{X^{\rm{ss}}(\ell)}
\]
and hence
\[
\Phi(\cV(0,-2)) = \cV(-1,-1) + \cV(2,2) + 2\cV(1,0) + \cV(0,0) - \cV(3,1) + \cV(3,3).
\]
\end{ex}

\end{rem}

\subsection{Extensions to categories of singularities and hyperk\"ahler quotients}

As in \S \ref{sect:hyperkaehler}, there is a version of \autoref{prop:representation} in which $\Pi_1(\cK)$ acts on the derived category of a hyperk\"{a}hler reduction of a symplectic linear representation, and more generally on the graded derived category of singularities ${\rm D}^b_{\op{sing}}(X/G\times \bG_m,W)$ for a graded LG model $W \colon X/G\times \bG_m \to \bA^1/\bG_m$. We assume throughout that the zonotope $\ol{\Sigma} \subset M_\bR$ associated to $X$ as a representation of $G$ spans $M_\bR$ and that $(M_\bR^W)_{\ol{\Sigma}\text{-{\rm gen}}} \neq \emptyset$.

In order to state our results, it is convenient to introduce a bit of machinery. Given a module category $\cC$ over the symmetric monoidal $\infty$-category $\Perf(\bA^1/\bG_m)^\otimes$, we define ${\rm D}_{\op{sing}}(\cC)$ to be the idempotent complete stable $\infty$-category obtained in the following way: First define $\cC_0$ to be the functor category $\cC_0 = \op{Fun}_{\Perf(\bA^1 /\bG_m)^\otimes} (\Perf(\{0\}/\bG_m), \cC)$. This category will have an endofunctor $ (-) \otimes \cO_{\bA^1}\langle 1 \rangle$ coming from the $\Perf(\bA^1/\bG_m)^\otimes$-module structure, and there is a natural transformation of endofunctors
\[
\beta \colon \op{id}_{\cC_0} \to (-) \otimes \cO_{\bA^1}\langle 1 \rangle[2].
\]
We define ${\rm D}_{\op{sing}}(\cC)$ to be the stable idempotent completion of the $\infty$-category whose objects are objects of $\cC_0$ and whose morphism spaces are
\[
\op{Map}_{{\rm D}_{\op{sing}}(\cC)}(E,F) = \op{hocolim}_p \op{Map}_{\cC_0}(E,F \otimes \cO_{\bA^1}\langle p \rangle [2p]),
\]
where the colimit is taken with respect to the natural transformation $\beta$. This construction induces a functor
\[
{\rm D}_{\op{sing}} \colon \op{Ho}(\Perf(\bA^1 / \bG_m)^\otimes \text{\rm{-Mod}}) \to \op{Ho}(\dgCat).
\]
For a more detailed and rigorous treatment of this construction, we refer the reader to \cite{Lurie}, which treats the analogous construction over the sphere spectrum rather than over a field $k$ of characteristic $0$ as we do here.

The content of the proof of \autoref{cor:sing} above is that if $W \colon X/G \times \bG_m \to \bA^1/\bG_m$ is a graded Landau--Ginzburg model, then if we consider the polytope $\ol{\nabla} \subset M_\bR$ corresponding to $X$ as a representation of $G$ and define $\cM(\delta + \ol{\nabla}) \subset {\rm D}^b(X/G\times \bG_m)$ by simply ignoring the $\bG_m$ action, then the restriction functor $\cM(\delta + \ol{\nabla}) \to {\rm D}^b(X^{\rm{ss}}(\ell)/G \times \bG_m)$ is an equivalence of $\Perf(\bA^1/\bG_m)^\otimes$-module categories when $\partial(\delta + \ol{\nabla}) \cap M = \emptyset$ and $X^{\rm{ss}}(\ell)=X^{\rm{s}}(\ell)$. Applying the construction ${\rm D}_{\op{sing}}(-)$ to the restriction functor then gives an equivalence
\[
{\rm D}_{\op{sing}}(\cM(\delta + \ol{\nabla})) \simeq {\rm D}_{\op{sing}}({\rm D}^b(X^{\rm{ss}}(\ell)/G\times \bG_m)) \simeq {\rm D}^b_{\op{sing}}(X^{\rm{ss}}(\ell)/G\times \bG_m).
\]
Furthermore, for any character $m$ of $G$, the functor $(\cO_X \otimes m) \otimes (-) \colon \cM(\delta + \ol{\nabla}) \to \cM(\delta+\ol{\nabla})$ is canonically a map of $\Perf(\bA^1/\bG_m)^\otimes$-module categories. Therefore the representation of \autoref{prop:representation} canonically lifts to a functor $F \colon \Pi_1(\cK_X) \simeq \tilde{\Gamma}(M^W,\cA) \to \op{Ho}(\Perf(\bA^1/\bG_m)^\otimes\text{\rm{-Mod}})$, and composing with ${\rm D}_{\op{sing}} \colon \op{Ho}(\Perf(\bA^1/\bG_m)^\otimes\text{\rm{-Mod}}) \to \op{Ho}(\dgCat)$ gives the following:

\begin{cor} \label{cor:action_on_singularities}
Let $X$ be a linear representation of $G \times \bG_m$ which is quasi-symmetric as a representation of $G$, and let $W \colon X \to \bA^1$ be a $G$-invariant and $\bG_m$-equivariant map. Then assigning $F(\delta) = {\rm D}_{\op{sing}}(\cM(\delta + \ol{\nabla}))$, assigning $F([\delta] \xrightarrow{\ell} [\delta'])$ to the functor
\[
{\rm D}_{\op{sing}}(\cM(\delta + \ol{\nabla})) \xrightarrow{\op{res}_{X^{\rm{ss}}(\ell)}} {\rm D}^b_{\op{sing}}(X^{\rm{ss}}(\ell)/G\times \bG_m,W) \xrightarrow{\op{res}_{X^{\rm{ss}}(\ell)}^{-1}} {\rm D}_{\op{sing}}(\cM(\delta + \ol{\nabla})),
\]
and assigning $F([\delta] \leadsto [\delta + m])$ to the functor
\[
{\rm D}_{\op{sing}}(\cM(\delta + \ol{\nabla})) \xrightarrow{(\cO_X \otimes m) \otimes (-)} {\rm D}_{\op{sing}}(\cM(\delta + \ol{\nabla}))
\]
defines a functor $F \colon \tilde{\Gamma}(M^W,\cA) \to \op{Ho}(\dgCat)$, where $\cA$ is the $M^W$-periodic hyperplane arrangement in $M_\bR^W$ determined by the polytope $\ol{\nabla} \subset M_\bR$. In particular, this defines a representation of $\Pi_1(\cK_X) \simeq \tilde{\Gamma}(M^W,\cA)$ under the equivalence of \autoref{prop:groupoid}.
\end{cor}

Finally, combining this corollary with the results of \autoref{sect:hyperkaehler} gives a groupoid action on the derived category of a hyperk\"ahler quotient. Applying the proof of \autoref{thm:hyperkaehler} verbatim gives the following:

\begin{cor} \label{cor:representation_hyperkahler}
Let $X$ be a symplectic linear representation of a split reductive group $G$ with moment map $\mu \colon X \to \fg^\dual$. Let $W \colon X \times \fg \to \bA^1$ be the $G$-invariant $\bG_m$-equivariant map induced by $\mu$. Regarding $X \oplus \fg$ as a quasi-symmetric representation of $G$ with corresponding zonotope $\ol{\Sigma}$ we assume that $\ol{\Sigma}$ spans $M_\bR$ and $(M_\bR^W)_{\ol{\Sigma}\text{-{\rm gen}}} \neq \emptyset$. Then for any $\ell$ for which $X_0^{\rm{ss}}(\ell) = X_0^{\rm{s}}(\ell)$ and $\delta$ such that $\partial(\delta + \ol{\nabla}) \cap M = \emptyset$, we have an equivalence
\[
{\rm D}_{\op{sing}}(\cM(\delta + \ol{\nabla})) \simeq {\rm D}^b(X_0^{\rm{ss}}(\ell)/G),
\]
through which \autoref{cor:action_on_singularities} gives an action of $\Pi_1(\cK_{X \oplus \fg})$ on ${\rm D}^b(X_0^{\rm{ss}}(\ell)/G)$.
\end{cor}

\begin{rem} \label{rmk:nabla=sigma}
A simple computation using \autoref{rmk:nabla-width} and the proof of \autoref{lem:polytopes} shows that the polytope $\ol{\nabla} \subset M_\bR$ associated to the quasi-symmetric representation $X \oplus \fg$ is actually equal to the zonotope $\ol{\Sigma}$ associated to the character of the representation $X$ itself (rather than $X \oplus \fg$). Therefore, the $M^W$-periodic hyperplane arrangement appearing in the definition of $\cK_{X \oplus \fg}$ is the locally periodic hyperplane arrangement associated to the condition $\partial(\delta + \ol{\Sigma}) \cap M = \emptyset$ by \autoref{lem:hyperplane}.
\end{rem}

\subsubsection{Example: Hilbert scheme of $n$ points in $\bC^2$} \label{sect:hilb}

Consider the Nakajima quiver variety associated to the quiver with one vertex and one loop at that vertex. We consider the dimension vector $\bv = (n)$ and $\bw = (1)$. In this case, the group is $G = {\bf GL}_n(k)$ and the linear representation $X \oplus \fg$ is
\[
\mathfrak{gl}_n \oplus T^*(\op{Hom}(k^n,k^n) \oplus \op{Hom}(k^n,k)).
\]

If we take a positive linearization for $G$, then the corresponding Nakajima quiver variety is the Hilbert scheme of $n$ points in $\bC^2$ \cite[Theorem 1.9]{Nakajima-book} (the description is slightly different, but see \cite[Proposition 2.8]{Nakajima-book} to see that they are equivalent).

Take $T$ to be the diagonal matrices and let $e_1, \dots, e_n$ be the standard basis for the character lattice, so that the weights of $X$ are $\{e_i, -e_i \mid 1 \le i \le n\}$ and $\{e_i - e_j \mid 1\le i,j\le n\}$ (the latter set appearing with multiplicity 2). The zonotope $\ol{\Sigma}$ of the character of $X$ is the Minkowski sum of the line segments $[-e_i,e_i]$ and $[2e_i-2e_j,2e_j-2e_i]$. Let $x_1, \dots, x_n$ be the dual basis to $e_1, \dots, e_n$. Also define $c_S = |S| + 2|S|(n-|S|)$ for any subset $S \subseteq \{1,\dots,n\}$. 

\begin{lem}
The defining inequalities for $\ol{\Sigma}$ are $|\sum_{i \in S} x_i| \le c_S$ where we range over all nonempty subsets $S$ of $\{1,\dots,n\}$. 
\end{lem}

\begin{proof}
Using \cite{mcmullen}, every facet contains a translate of the Minkowski sum of $n-1$ linearly independent vectors from the set of $e_i$ and $e_i-e_j$ and so its defining equation vanishes on these vectors. If the vectors contain $e_{i_1}, \dots, e_{i_r}$ and no other $e_j$, then the equation must be $\sum_{i \in S} x_i$ where $S = \{1,\dots,n\} \setminus \{i_1,\dots,i_r\}$. So all facets are of this form.

Conversely, we claim that $\sum_{i \in S} x_i \le c_S$ vanishes on a facet. The function $\sum_{i \in S} x_i$ takes on the value $c_S$ on the sum $v_S = \sum_{i \in S} e_i + \sum_{i \in S} \sum_{j \notin S} (2e_i - 2e_j)$ and it is clear this is the maximal value that it takes. Set $S = \{i_1,\dots,i_r\}$ and consider the set $\{v_S + e_j \mid j \notin S\} \cup\{v_S + e_{i_j} - e_{i_{j+1}} \mid j=1,\dots,r-1\}$. This is a linearly independent set of $n-1$ vectors and hence its convex hull is contained in a facet of $\ol{\Sigma}$. So $\sum_{i \in S} x_i \le c_S$ is a defining inequality. By symmetry, we conclude that $\sum_{i \in S} x_i \ge -c_S$ is also a defining inequality.
\end{proof}

By \autoref{rmk:nabla=sigma}, the polytope $\ol{\nabla}$ associated to $X \oplus \fg$ is the zonotope $\ol{\Sigma}$ associated to $X$. So the periodic locally finite arrangement in \autoref{lem:hyperplane} is given by the hyperplanes $\sum_{i \in S} x_i = c$ where $c$ is an arbitrary integer and $S$ is a nonempty subset of $\{1, \dots, n\}$ (translating by $e_i$ changes $c_S$ by 1 each time). The subspace $M_\bR^W$ is the line spanned by $e_1 + \cdots + e_n$, so the intersection with this arrangement gives the set $\{\frac{a}{b} (e_1 + \cdots + e_n) \mid a,b \in \bZ,\ 1 \le b \le n\}$.

The group $M^W$ is also generated by $e_1+\cdots+e_n$, so the complexified K\"ahler moduli space is $\cK_{X \oplus \fg} = (\bC \setminus (\bZ \cup \frac{1}{2} \bZ \cup \cdots \cup \frac{1}{n} \bZ)  ) / \bZ$, which under the exponential map $q = \exp(2\pi i (\delta + i\ell))$ is identified with
\[
\cK_{X \oplus \fg} = \{ q \in \bC^\ast \mid q^k \neq 1, \forall  k = 1,\ldots,n \}.
\]

\end{document}